\documentclass[aop,preprint]{imsart}

\RequirePackage[OT1]{fontenc}
\RequirePackage{amsthm,amsmath,amssymb}
\RequirePackage[numbers]{natbib}
\RequirePackage[colorlinks,citecolor=blue,urlcolor=blue]{hyperref}

\usepackage{fullpage}

\usepackage{color, graphicx}
\usepackage{lmodern} 
\usepackage[latin1]{inputenc}
\usepackage[numbers]{natbib}
\usepackage{mathrsfs}
\usepackage{booktabs}
\usepackage{url}
\usepackage{amscd}
\usepackage{amstext}
\usepackage{amsfonts}
\usepackage{mathrsfs}
\usepackage{bbm}
\usepackage{soul}
\usepackage[T1]{fontenc}
\numberwithin{equation}{section}
\usepackage{enumerate}
 \usepackage{epstopdf}
\usepackage{float}
\usepackage{mathtools} 
\mathtoolsset{showonlyrefs}
\usepackage{todonotes}
\usepackage{natbib} 

\startlocaldefs

\newtheorem{theorem}{Theorem}[section] 
\newtheorem{lemma}[theorem]{Lemma} 
\newtheorem{proposition}[theorem]{Proposition}
\newtheorem{corollary}[theorem]{Corollary}
\newtheorem{definition}[theorem]{Definition}
\newtheorem{remark}[theorem]{Remark}

\newtheorem{example}[theorem]{Example}

\newcommand{\eqnsection}{
\renewcommand{\theequation}{\thesection.\arabic{equation}}
 \makeatletter   \csname  @addtoreset\endcsname{equation}{section}
   \makeatother}
   
\renewcommand{\mathcal}{\mathscr}

\def \np{\par\noindent}

\def \qed{$\Box $}

\def\R{\mathbb{R}} 
\def\N{\mathbb{N}}

\def\E{\mathbb{E}} 

\def\P{\mathbb{P}}

\def\cB{\mathcal{B}} 
\def\cF{\mathcal{F}} 
 
\def\cL{\mathcal{L}} 
 
\def\cS{\mathcal{S}}

\def\0{\mathbf{0}} 
\def\1{\mathbf{1}}

\def\var{{\rm Var}}

\def \eid{\stackrel{d}{=}} 
\def \cip{\stackrel{P}{\rightarrow}} 

\def \lt{[\hskip-1.3pt[ }
\def \rt{ ]\hskip-1.3pt]}
\def\({ \left( }
\def\){\right) }

\def \hT{\hat{T}}

\newcommand{\devnull}[1]{}

\endlocaldefs

\begin{document}
\begin{frontmatter}
\title{Representations and isomorphism identities for  infinitely divisible processes}
\runtitle{Isomorphisms for  infinitely divisible processes}
\author{\fnms{Jan} \snm{Rosi\'nski}\thanksref{t1}\ead[label=e1]{rosinski@utk.edu}}
\runauthor{J. Rosi\'nski}
\affiliation{University of Tennessee}
\address{Department of Mathematics\\
University of Tennessee\\
Knoxville, Tennessee 37996\\
 USA\\ 
\printead{e1}}
\thankstext{t1}{Research partially supported by the Simons Foundation grant 281440}

%




\date{November 19, 2017}

\maketitle

\begin{abstract}
\noindent 
We propose isomorphism type identities for nonlinear functionals of general infinitely divisible processes. Such identities can be viewed as an analogy of the Cameron-Martin formula for Poissonian infinitely divisible processes but with random translations. The applicability of such tools relies on precise understanding of L\'evy measures of  infinitely divisible processes and their representations, which are studied here in full generality. We illustrate this approach on examples of squared Bessel processes, Feller diffusions, permanental processes, as well as L\'evy processes. 
\end{abstract}



\end{frontmatter}

\smallskip



\eqnsection

\baselineskip16pt

\section{Introduction}\label{s:intro}

Let $G=(G_t)_{t\in T}$ be a centered Gaussian process over an arbitrary set $T$. The Cameron-Martin Formula says that for every random variable $\xi$ in the $L^2$-closure of the subspace spanned by $G$ and for any  measurable functional $F: \R^T \mapsto \R$ 
\begin{equation} \label{C-M}
  \E \left[ F\left((G_t + \phi(t))_{t\in T}  \right) \right] = \E \left[ F\left((G_t)_{t\in T}\right) e^{\xi- \frac{1}{2}\E \xi^2}\right]
\end{equation}
where $\phi(t)= \E(\xi G_t)$. This formula has many applications, including SDEs and SPDEs driven by Gaussian random fields. 
It can also be viewed as an isomorphism identity between functionals of a translated Gaussian process and the corresponding functionals of the untranslated process,  under the changed probability measure. 

It is well-known that \eqref{C-M} does not extend to the Poissonian case. Indeed, it is easy to show that if $Y=(Y_{t})_{t\in [0,1]}$ is a Poisson process, then there is no function $\psi: [0,1] \to \R$, $\psi \not\equiv 0$, such that 
\begin{equation} \label{C-M-P}
  \E \left[ F\left((Y_t + \psi(t))_{t\in [0,1]}  \right) \right] = \E \left[ F\left((Y_t)_{t\in [0,1]}\right) \eta \right]
\end{equation}
for all measurable functionals $F: \R^{[0,1]} \mapsto \R$ and some random variable $\eta \ge 0$ with $\E \eta=1$. 

In this paper we propose isomorphism identities based on random translations as follows. Let $X=(X_{t})_{t\in T}$ be an infinitely divisible process  over a general set $T$ (i.e., a process whose finite dimensional distributions are infinitely divisible).  
Then, for every process $Z= (Z_{t})_{t\in T}$ independent of $X$, whose distribution is absolutely continuous with respect to  L\'evy measure of $X$, there exists a measurable function $g: \R^T \mapsto \R_{+}$ such that for any measurable functional $F: \R^T \mapsto \R$
\begin{equation} \label{JR}
  \E \left[ F\left((X_t + Z_t)_{t\in T}  \right) \right] = \E \left[ F\left((X_t)_{t\in T}\right)\,  g(X) \right]\, .
\end{equation} 

 Relation \eqref{JR} was inspired by Dynkin's Isomorphism Theorem \cite{Dynkin84}, where $Z=(Z_x)_{x \in E}$ is the total accumulated local time at a state $x$ of a strongly symmetric transient Markov process with the state space $E$ and $X=(X_x)_{x \in E}$ is the squared associated Gaussian process, see \citet[Chapter 8]{Marcus-Rosen(book)}. Since the process $X$ is infinitely divisible by results of \citet{Eisenbaum03} and  \citet{Eisenbaum-Kaspi06}, and $Z$ (over $T=E$) satisfies assumptions of \eqref{JR}, Dynkin's Isomorphism Theorem holds as a special case of \eqref{JR}.  See Example \ref{e:Dynkin} for details. Actually, we present this example in a more general setting of permanental processes, following work of \citet{Eisenbaum-Kaspi09}. 
 
There are two basic directions of applying identity \eqref{JR}. The first one is to start with a process $Z= (Z_{t})_{t\in T}$ of interest, associate with it (possibly) easier to handle infinitely divisible process $X=(X_{t})_{t\in T}$ whose L\'evy measure dominates the law of $Z$, and transfer path properties of $X$ to $Z$ via isomorphism \eqref{JR}.  Using Dynkin's Isomorphism Theorem, Marcus and Rosen derived many results for local times of Markov processes, including L\'evy processes, see, e.g.,  \cite{Marcus-Rosen92}, \cite{Marcus-Rosen(book)}. 
Another direction of applications of \eqref{JR} is much harder, to derive information about $X$ by utilizing $Z$. One way to approach it is to consider the ``converse'' version of \eqref{JR} which  expresses $X$ as the process $X+Z$ with changed measure.  Proposition \ref{p:Lp0} is proven in this spirit.  In Section \ref{s:iso} and \ref{s:series} we present direct and converse versions of \eqref{JR}.

In Section \ref{s:series} we link isomorphism identity \eqref{JR} with series representations of infinitely divisible processes.  This yields more insight into the structure of the admissible translation $Z$, which could be viewed as the zero-term in the series expansion of  $X$, see Theorem \ref{t:iso-s}. Such representations have proven useful in the study of path regularities of infinitely divisible processes, see, e.g., \citet[ch. 11]{Talagrand(14)}.

Successful implementation of isomorphism identities requires precise understanding of L\'evy measures of processes, which are defined on path spaces with the usual cylindrical $\sigma$-algebras (as opposed to $\sigma$-rings in \cite{Lee67} and \cite{Maruyama70}). Section \ref{s:lm} contains systematic development of L\'evy measures and spectral representations based on lecture notes \cite{Rosinski07}. We view L\'evy measures as  ``laws of processes'' defined on possibly infinite measure spaces and call such ``processes'' representations of L\'evy measures. Properties of L\'evy measures are defined by properties of their representations. Transfer of regularity property (Theorem \ref{transfer}) puts the L\'evy measure on the same Borel function space where paths of the corresponding infinitely divisible processes belong. This allows to relate path properties of processes and representations of their L\'evy measures. 

Throughout this paper we illustrate general concepts by selected examples of infinitely divisible processes. For this purpose we have chosen L\'evy processes, squared Bessel processes, Feller diffusions, general compound Poisson processes, and permanental processes. Some of these processes are presented in their simplest forms in order to not obscure ideas by technical complications. This list of examples can be extended by many processes of interest, including cylindrical L\'evy processes in Banach spaces of \cite{A-R(10)}, which can be considered as infinitely divisible processes defined on $T= \R \times E'$, where $E$ is a Banach space, selfdecomposable fields \cite{BSS(15)}, and multidimensional infinitely divisible processes, which become one-dimensional after enlarging the index set.

The paper is organized as follows. Section \ref{s:lm} presents a detailed study of L\'evy measures on path spaces. Theorem \ref{t:LK-P} shows existence and uniqueness of such measures for every Poissonian infinitely divisible processes. Theorem \ref{t:cLM} characterizes the $\sigma$-finiteness of L\'evy measures and Theorem \ref{n-sf} characterizes Poissonian infinitely divisible process having not $\sigma$-finite L\'evy measures. The concept of representations of L\'evy measures is introduced and discusses on examples in the final part of Section \ref{s:lm}. Section \ref{s:LI} gives a L\'evy-It\^o representation for general infinitely divisible processes (see Theorem \ref{t:LI}) and Theorem \ref{transfer} states the above mentioned transfer of regularity property. In Section \ref{s:iso} isomorphism identities are given in Theorems \ref{t:iso0}, \ref{t:iso1}, and \ref{t:iso2}. 
Dynkin Isomorphism Theorem for permanental processes is discussed in Example \ref{e:Dynkin}.  Proposition \ref{p:s-iso} characterizes processes satisfying an abstract form of Dynkin's isomorphism, first considered in \cite[Lemma 3.1]{Eisenbaum08}. In the final part of this section, applications of isomorphism identities for L\'evy processes are given. Section \ref{s:series} relates representations of L\'evy measures of Section \ref{s:lm} to series representations of Poissonian infinitely divisible processes. As examples we give series representations of Feller diffusions and squared Bessel processes. Then we connect such representations with isomorphism identities in Theorem \ref{t:iso-s}. Section~\ref{s:Proofs} contains the proofs of results from Sections \ref{s:lm}--\ref{s:series}. 

Recall that a stochastic process $X=\(X_t\)_{t \in T}$, with an arbitrary index set  $T$, is said to be Poissonian infinitely divisible if its all finite dimensional marginal distributions are infinitely divisible  without Gaussian part. 
 
Throughout this paper, an identity as \eqref{JR} reads: if one side exists then the other does and they are equal.

\bigskip

\section{L\'evy measures on path spaces}\label{s:lm}

\subsection{Definitions and preliminaries}\label{s:pre}

L\'evy measures of probability laws on $\R^d$ are usually defined either on $\R^d \setminus \{0 \}$ or on  $\R^d$, in the second case under the assumption that L\'evy measures do not charge the origin. Identifying $\R^d$ with $\R^T$, where $T=\{1,\dots,d \}$, we have two natural ways to define L\'evy measures for infinitely divisible processes over any set $T$. The first way is to define a L\'evy measure on the $\sigma$-ring generated by cylindrical subsets of $\R^T \setminus \{0\}$, as proposed by \citet{Lee67} and \citet{Maruyama70}. This approach, however, leads to substantial conceptual and technical difficulties when $T$ is uncountable. Therefore, we have chosen the second way, to consider L\'evy measures on the canonical path space $(\R^T, \cB^T)$, on which the laws of stochastic processes over $T$ are defined. This approach, in particular,  allows us to talk about L\'evy measures as ``laws of stochastic processes''. 

Let $\R^T$ be the space of all functions $x: T \mapsto \R$, and let  $\cB^T$ denote its cylindrical (product) $\sigma$-algebra. 
The law of a stochastic process $X=(X_{t})_{t\in T}$ is a probability measure $\mu$ on $(\R^{T}, \cB^{T})$ given by
\begin{equation} \label{}
  \mu(A) = \P\{\omega: (X_{t}(\omega))_{t\in T} \in A\}, \quad A \in \cB^{T},
\end{equation}
and we write  $\cL(X)=\mu$. For any $S \subset T$, $X_S :=(X_{t})_{t\in S}$ is the restriction of the process $X$ to the index set $S$. Similarly, for any $x \in \R^T$, $x_S$  denotes the restriction of function $x$ to $S \subset T$. Finally, $0_S$  stands for the origin of $\R^S$, which will be viewed as a point or the one-point set, depending on the context.

\begin{definition}
A measure $\nu$ on $(\R^{T}, \cB^{T})$ is said to be a L\'evy measure if the following two conditions hold
\begin{enumerate}[\rm (L1)]
  \item\label{LM} for every $t \in T$ \ 
   $\int_{\R^{T}} |x(t)|^2 \wedge 1 \, \nu(dx)< \infty$,  
  \item\label{LM0} for every  $A \in \mathcal{B}^T$ \ 
    $\nu(A) = \nu_{\ast}(A \setminus 0_T)$, \  
    where $\nu_{\ast}$ is the inner measure.
\end{enumerate}
\end{definition}

\medskip

The first condition is a technical one, needed for the integral in the L\'evy-Khintchine formula \eqref{LK} to be well-defined. The second condition gives the meaning to ``$\nu$ does not charge the origin''.
If $T$ is countable, then $0_T \in \mathcal{B}^T$ and (L\ref{LM0}) is equivalent to $\nu(0_T)=0$, which is the usual condition for L\'evy measure.  If $T$ is uncountable, then $0_T \notin \mathcal{B}^T$, so that $\nu(0_T)$ is undefined. However, (L\ref{LM0}) still makes sense. We will show that every infinitely divisible process has a unique measure satisfying this definition.

\begin{remark}\label{r:cLM}
{\rm 
The following condition implies (L\ref{LM0}) and is often easier to check: there exists a countable set $T_0 \subset T$ such that 
\begin{equation} \label{cLM}
   \nu\{x\in \R^T: x_{T_{0}} = 0\}=0\, .
\end{equation}
Indeed, from \eqref{cLM} we get for any $A \in \cB^T$ 
$$
\nu(A) \ge  \nu_{\ast}(A \setminus 0_T) \ge \nu(A \setminus \{x: x_{T_{0}} =0\}) = \nu(A)\, ,
$$
which shows (L\ref{LM0}). 
}
\end{remark}

\bigskip

Throughout this paper, $\hT$ will denote the family of all finite nonempty subsets of the index set $T$, 
\begin{equation} \label{}
  \hT =\{I \subset T: 0< \mathrm{Card}(I)< \infty \},
\end{equation}
so that for any  $I \in \hT$, \ $\R^I$ can be identified with the Euclidean space $\R^{\mathrm{Card}(I)}$. We also set
$$
\hT_c := \{J \subset T:   \ \text{$J$ is nonempty countable} \}. 
$$
Let ${\pi_S: \R^T \mapsto \R^S}$ is the projection from $\R^T$ onto $\R^S$, $\pi_S(x) :=x_S$. Finally,
$$
\cB_0^S := \{B \in \cB^S: 0_S \notin B \}. 
$$

The next lemma sheds more light on condition (L\ref{LM0}). 
 
\begin{lemma}\label{l:nu*}
	Let $\nu$ be an arbitrary measure on $(\R^T, \cB^T)$. Then for every $ A \in \cB^T$
	\begin{equation} \label{nu*}
  \nu_{*}(A \setminus 0_T) = \sup_{J \in \hT} \nu(A \setminus \pi^{-1}_J(0_J) ) \, .
  \end{equation}
Consequently, 
$$
\nu^0( A ) := \nu_{*}(A \setminus 0_T), \quad A \in  \cB^T
$$ 
is a measure satisfying (L\ref{LM0}) and  
\begin{equation} \label{nu_0}
  \nu^0 = \nu \quad \text{on } \ \cB^T_0\, .
\end{equation}
Therefore, if $\nu$ satisfies (L\ref{LM}) then $\nu^0$ is a L\'evy measure. 
\end{lemma}

\medskip

Below we give some equivalent conditions to (L\ref{LM0}) that can be easier to verify.

\begin{lemma}\label{LM-equiv}
Let $\nu$ be a measure on $(\R^{T}, \cB^{T})$. The following conditions are equivalent to  (L\ref{LM0}). 
\begin{itemize}
  \item[\rm (a)] for every $T_0 \in \hT_c$ there exists  $T_1 \in \hT_c$ such that $T_0 \subset T_1$ and
 $$
 \nu\{x\in \R^T: x_{T_{0}} = 0\}=\nu\{x\in \R^T: x_{T_{0}} = 0, \ x_{T_{1}} \ne 0\}.
$$
\item[\rm (b)] for every  $T_0  \in \hT_c$ with $\nu\{x\in \R^T: x_{T_0} = 0\} > 0$ there is $t \notin T_0$ such that 
\begin{equation} \label{unc1}
  \nu\{x\in \R^T: x_{T_0} = 0, \ x(t) \ne 0 \} > 0;
\end{equation}
\item[\rm (c)] either \eqref{cLM} is satisfied for some  $T_0  \in \hT_c$  or for every  $T_0  \in \hT_c$ there is $t \notin T_0$ such that \eqref{unc1} holds.
\end{itemize}
\end{lemma}

\medskip 

\begin{remark}
{\rm
	Condition (a) was the original condition for a L\'evy measure  in \citet{Rosinski07}. 
Condition (b) was communicated to us as equivalent to (a) by Gennady Samorodnitsky. Notice a subtle difference between (b) and the second alternative condition in (c). 
}
\end{remark}


\bigskip

\subsection{L\'evy-Khintchine and canonical spectral representations}\label{s:LK}

 
Let $X=\(X_t \)_{t \in T}$ be an infinitely divisible process, so that for every $I \in \hT$ the random vector $X_I$ is infinitely divisible in $\R^I$  (which is considered as $\R^{\mathrm{Card}(I)}$ with the inner product $\langle \cdot , \cdot \rangle$ and the norm $|\cdot|$). By the L\'evy-Khintchine representation \cite[Theorem 8.1] {Sato(99)}, there exists a unique triplet $(\Sigma_I, \nu_I, b_I)$ such that for every  $a \in \R^I$ 
\begin{equation} \label{LK-fidi}
  \E \exp i  \langle a, X_{I}\rangle =  \exp \left\{-\frac{1}{2} \langle a, \Sigma_I a \rangle  + i \langle a, b_I \rangle + \int_{\R^I} (e^{\langle a, x  \rangle } - 1 - i \langle a, \lt x\rt \rangle) ~\nu_I(dx)  \right\},
\end{equation}
where $\Sigma_I$ is a non-negative definite $I \times I$-matrix, $b_I\in \R^I$, and $\nu_I$  a L\'evy measure, i.e., $\nu_I$ is a Borel measure on $\R^I$ satisfying 
\begin{equation} \label{LM-fidi}
\int_{\R^I} |x_I|^2 \wedge 1 \, \nu_I(dx)< \infty \quad \hbox{and} \quad \nu_I(0_I)=0\, .
\end{equation}
Here $\lt \cdot \rt$  is a fixed truncation function defined as follows. Let $\chi: \R \mapsto \R$ be a bounded measurable function such that $\chi(v)=1+o(|v|)$ as $v \to 0$ and  $\chi(v)=O(|v|^{-1})$ as $|v| \to \infty$. We will call $\chi$ a cutoff function. For example, functions $\1_{\{|v|\le 1 \}}$, $(1\vee |v|)^{-1}$, $(1 + |v|^2)^{-1}$, and the usual cutoff function in a neighborhood of 0 satisfy these conditions.
The truncation of $v=(v_1,\dots,v_n) \in \R^n$ is defined by
\begin{equation} \label{truncation}
  \lt v\rt = \left( v_1\chi(|v_1|), \dots, v_n\chi(|v_n|) \right)\, .
\end{equation}
Similarly, if $y \in \R^S$ then $\lt y \rt \in \R^S$ is given by $\lt y\rt(s) = y(s) \chi(y(s)) $, $s \in S$. 

From the uniqueness of the triplet $(\Sigma_I, \nu_I, b_I)$ in \eqref{LK-fidi}, the following consistency conditions hold: for every $I, J \in \hT$ with  $I \subset J$ 
\begin{enumerate}[(c1)]
\item  $\Sigma_{J}$ restricted to $I \times I$ equals $\Sigma_{I}$, \label{c1}
  \item $b_J$ restricted to $I$ equals $b_I$, \label{c2}
  \item $\nu_J \circ \pi_{IJ}^{-1} = \nu_I$ on  \ $\cB^I_0$, \label{c3}
\end{enumerate} 
where $\pi_{IJ}: \R^J \mapsto \R^I$  denotes the natural projection from $\R^J$ onto $\R^I$. By the Kolmogorov Extension Theorem, there exist mutually independent centered Gaussian process $G=(G_{t})_{t\in T}$  and a Poissonian infinitely divisible process $Y=(Y_{t})_{t\in T}$ such that 
\begin{equation} \label{G+Y}
  X \eid G + Y\, ,
\end{equation}
where for every $I \in \hT$,  $G_I \sim N(0, \Sigma_{I})$ and 
\begin{equation} \label{LK-fidi-P}
  \E \exp i  \langle a, Y_{I}\rangle =  \exp \left\{ i \langle a, b_I \rangle + \int_{\R^I} (e^{\langle a, y  \rangle } - 1 - i \langle a, \lt y\rt \rangle) ~\nu_I(dy)  \right\}, \quad a \in \R^I.
\end{equation}
The covariance function $\Sigma$ of $G$  restricted to $I \in \hT$, equals $\Sigma_I$; similarly, by {\rm (c\ref{c2})} there is a path $b: T \mapsto \R$ whose restrictions to $I$ coincide with $b_I$. 


\begin{definition}
	We say that a family $\{\nu_I:  I \in \hT \}$ of finite dimensional L\'evy measures is consistent when it satisfies condition {\rm (c\ref{c3})}.
\end{definition}
  
  \medskip
  
  It should be noted that a consistent family of finite dimensional L\'evy measures is not necessarily a projective system.

\begin{example}
Let $T=\N$ and let $X=\{X_n \}_{n\in \N}$ be an i.i.d. sequence of Poisson random variables with mean 1. The L\'evy measure of $X_{\{1,\dots,n \}}=(X_1,\dots,X_n)$ is given by
	$$
	\nu_{\{1,\dots,n \}} = \sum_{k=1}^{n} \delta_{0_{\{1,\dots,k-1\}}} \otimes \delta_1 \otimes \delta_{0_{\{k+1,\dots,n\}}}\, .
	$$
	Therefore, if $I= \{1,\dots,n \}$ and $J=\{1,\dots,r \}$ with $I \subset J$,
$$
\nu_J \circ \pi_{IJ}^{-1} = \nu_I + (r-n) \delta_{0_I}\, , 
$$	
which shows that $\{\nu_I:  I \in \hT \}$ is not a projective system of measures.	
\end{example}

This fact makes ``glueing together'' $\nu_I$'s  more complicated than it would be for projective systems. Nevertheless, we have the following. 

\bigskip

\begin{theorem}\label{t:LK-P}
	Let $Y=(Y_{t})_{t\in T}$ be a Poissonian infinitely divisible process as in \eqref{LK-fidi-P}.  Then there exist a unique L\'evy measure $\nu$ on $(\R^T, \cB^T)$ and a shift function $b \in \R^T$ such that for every $I \in  \hT$ and $a \in \R^I$
\begin{equation}
	\label{LK-P}
  \E \exp i \sum_{t\in I} a_t Y_{t} \\
  =  \exp \left\{\int_{\R^T} (e^{i\langle a,x_I \rangle } - 1 - i \langle a, \lt x_I\rt \rangle) ~\nu(dx) + i \langle a, b_I \rangle  \right\}\, .
\end{equation} 
Therefore, for any consistent system of L\'evy measures $\{\nu_I:  I \in \hT \}$ there exists a unique L\'evy measure $\nu$ on $(\R^T, \cB^T)$ such that 
\begin{equation} \label{LM-cons}
  \nu \circ \pi_I^{-1} = \nu_I \quad \text{on }  \cB^I_0, \ I \in \hT.
\end{equation}
Furthermore, if $\rho$ is a measure such that $\rho \circ \pi_I^{-1} = \nu_I$ on $ \cB^I_0$ for all $I\in \hT$, then 
$\nu=\rho^0 \le \rho$. \\ (Cf. Lemma \ref{l:nu*}.)  
\end{theorem}

\bigskip

\begin{corollary}[L\'evy-Khintchine representation]\label{LKR}
Let $X=\(X_t \)_{t \in T}$ be an infinitely divisible process. Then there exist a unique triplet $(\Sigma, \nu, b)$ consisting of a non-negative definite function $\Sigma$ on $T \times T$, a L\'evy measure $\nu$ on $(\R^T, \cB^T)$  and a function $b\in \R^T$ such that for every $I \in \hat T$ and $a \in \R^I$
\begin{equation} \label{LK}
  \E \exp i \sum_{t\in I} a_t X_{t} =  \exp \left\{-\frac{1}{2} \langle a, \Sigma_I a \rangle  + \int_{\R^T} (e^{i\langle a,x_I \rangle } - 1 - i \langle a, \lt x_I\rt \rangle) ~\nu(dx) + i \langle a, b_I \rangle  \right\}\, ,
\end{equation}
where $\Sigma_I$ is the restriction of $\Sigma$ to $I \times I$.  $(\Sigma, \nu, b)$ is called the generating triplet of $X$. Conversely, given a generating triplet $(\Sigma, \nu, b)$ as above, there exists an infinitely divisible process $X=\(X_t \)_{t \in T}$ satisfying \eqref{LK}.	
\end{corollary}

Let $Y=(Y_{t})_{t\in T}$ be a Poissonian infinitely divisible process with L\'evy measure $\nu$. Let $N$ be a Poisson random measure on $(\R^T, \cB^T)$ having  intensity measure $\nu$. The existence of such $N$ follows from Kolmogorov's Extension Theorem. We will work with a stochastic integral of the form
\begin{equation} \label{int}
  I_N(f) = \int_{\R^T} f(x) [N(dx) - \chi(f(x)) \nu(dx)]\, ,
\end{equation}
where $\chi$ is a cutoff function defined at the beginning of this section. 
Since $\nu$ is not necessarily $\sigma$-finite, we take an extra care in handling  this integral. 

At the outset notice that the present development works for a Poisson random measure on any measure space $(S, \cS, n)$.  For the sake of concreteness we take here $(\R^T, \cB^T, \nu)$  and $\chi(v)=\1_{\{|v| \le 1 \}}$. Let $f = \sum_{j=1}^{n} a_j \1_{A_j}$, where $a_j \in \R$ and $A_j$ are disjoint with $\nu(A_j)< \infty$. We have 
$$
I_N(f) = \sum_{j \in J_0} a_j[N(A_j) - \nu(A_j)] + \sum_{j \in J_1} a_j N(A_j) = S_0+S_1
$$
where $J_0=\{j \le n:  |a_j|\le 1\}$ and $J_1=\{j \le n:  |a_j| > 1\}$. Hence
\begin{align*} 
  \E (|I_N(f)| \wedge 1 ) &\le \E (|S_0| \wedge 1 ) + \E (|S_1| \wedge 1 )
  \le (\E |S_0|^2 )^{1/2} + \E (|S_1| \wedge 1 ) \\
  & \le (\sum_{j \in J_0} a_j^2 \nu(A_j))^{1/2} + \sum_{j \in J_1} \E (|a_j N(A_j)|\wedge 1) \\
  & \le ( \int |f|^2 \wedge 1 \, d\nu)^{1/2} + \sum_{j \in J_1} \E[ (|a_j|^2\wedge 1) N(A_j) ] \\
  & \le ( \int |f|^2 \wedge 1 \, d\nu)^{1/2} +  \int |f|^2 \wedge 1 \, d\nu\, .
\end{align*}
Since $\E (|I_N(f)| \wedge 1 ) \le 1$, we infer that
\begin{equation} \label{PoisInt}
  \E (|I_N(f)| \wedge 1 ) \le 2 \( \int |f|^2 \wedge 1 \, d\nu\)^{1/2}.
\end{equation}
A change of the cut-off function $\chi$ will result in a change of constant  2  to another universal constant. 
Therefore, the integral in \eqref{int} is well-defined for any measurable $f: \R^T \mapsto \R$  such that $ \int |f|^2 \wedge 1 \, d\nu < \infty$ by the standard approximation procedure.  $I_N(f)$  is a Poissonian infinitely divisible random variable with the characteristic function 
\begin{equation} \label{int-cf}
  \E \exp\(i \theta I_N(f) \) = \exp\( \int_{\R^T} \(e^{i\theta f(x)} - 1 - i \theta f(x) \chi(f(x)) \)  \, \nu(dx)\). 
\end{equation}
 
 \bigskip
 
\begin{proposition}\label{p:can-spe-rep}
	Let $Y=(Y_{t})_{t\in T}$ be a Poissonian infinitely divisible process with L\'evy measure $\nu$ and a shift function $b$. Let $N$ be a Poisson random measure on $(\R^T, \cB^T)$ having  intensity measure $\nu$. Then the process $\widetilde{Y}= (\widetilde{Y}_{t})_{t\in T}$ given by
\begin{equation} \label{can}
  \widetilde{Y}_t = \int_{\R^T} x(t) [N(dx) - \chi(x(t)) \nu(dx)] + b(t), \quad t \in T
\end{equation}
has the same distribution as $Y$. $\widetilde{Y}$ will be called a canonical spectral representation of $Y$. 
\end{proposition}

\bigskip


\subsection{Sigma-finiteness of L\'evy measures}\label{s:sf}

The $\sigma$-finiteness of measures is an important property but not every L\'evy measure is $\sigma$-finite. This is shown in the following simple example. 
\medskip

\begin{example}\label{IPois}
{\rm 
Let $X=(X_t)_{t\in T}$ be un uncountable family of independent Poisson random variables with parameter 1. 
Then, for every $I \in \hat T$,  $a \in \R^I$,
\begin{equation} 
  \E \exp i\sum_{t \in I} a_t X_{t} = \exp \left[\sum_{t \in I} (e^{i a_{t} } - 1) \right] = 
  \exp \left[ \int_{\R^T} (e^{i \langle a, x_I \rangle } - 1)\, \nu(dx) \right],
\end{equation}
where $\nu$ is the counting measure of a set $E$ given by
$$
E = \{e_s \in \R^T:  s \in T, \ e_s(t)=1 \ \text{if} \ t=s \ \text{and} \ e_s(t)=0 \ \text{otherwise}\}.  
$$
We have $\int_{\R^T} |x(t)|^2 \wedge 1 \, \nu(dx) = |e_t(t)|^2 =1$ for every 
$t \in T$.  For $T_0 \in \hT_c$ choose $t \notin T_0$, and consider $A= \{x:  x_{T_0} = 0,  x(t) \ne 0\}$. Since $A \cap E = \{e_t\}$, we have $\nu(A)=1>0$, so that $\nu$ is L\'evy measure of $X$ by Lemma \ref{LM-equiv}(c).   
However, $\nu$ is not $\sigma$-finite as the counting measure  of an uncountable set.
}
\end{example}

\bigskip

The next theorem gives criteria when a L\'evy measure is $\sigma$-finite. Notice a subtle difference between (L\ref{LM0}) and $(ii)$.

\begin{theorem}\label{t:cLM}
Let $\nu$ be a L\'evy measure on $(\R^T, \cB^T)$. The following are equivalent:
\begin{itemize}
  \item[{ (i)}]  $\nu$ is  $\sigma$-finite;
  
   \item[{ (ii)}] $\nu^{\ast}(0_T)=0$, where $\nu^{\ast}$ is the outer measure;
   
  \item[{ (iii)}] $   \nu\{x\in \R^T: x_{T_{0}} = 0\}=0$ for some  $T_0  \in \hT_c$ (i.e., \eqref{cLM} holds).
\end{itemize} 
\end{theorem}

\bigskip

\begin{corollary}\label{c:unc}
	A L\'evy measure is not $\sigma$-finite if and only if for every  $T_0 \in \hT_c$ there exists $t \notin T_0$ such that $\nu\{x: x_{T_0} = 0, \ x(t)\ne 0\}>0$.
\end{corollary}
\np
{\bf Proof:} Condition (c) of Lemma \ref{LM-equiv} divides L\'evy measure into two categories: those which satisfy \eqref{cLM}, which are $\sigma$-finite by Theorem \ref{t:cLM}, and the others which satisfy the condition of this corollary. 
\qed

\bigskip

We may ask what Poissonian processes do not have $\sigma$-finite L\'evy measures? The next theorem characterizes them.

\begin{theorem}\label{n-sf}
Let $Y=(Y_{t})_{t\in T}$ be a Poissonian infinitely divisible process with L\'evy measure $\nu$. Then $\nu$ is not $\sigma$-finite if and only if $T$ is uncountable and there is a version $\widetilde{Y}= (\widetilde{Y}_t)_{t \in T}$ of the process  $Y$ such that for every  $T_0  \in \hT_c$ there exist $t_1 \notin T_0$ and independent random variables $\xi$ and $\eta$ such that 
	\begin{itemize}
  \item[(a)] $\widetilde{Y}_{t_1} = \xi + \eta$; \smallskip
    \item[(b)]  $(\widetilde{Y}_t, \xi, \eta: t \in T_0)$ are jointly Poissonian infinitely divisible; \smallskip
  \item[(c)] $\eta$ is non-degenerate and independent of $(\widetilde{Y}_t, \xi: t\in T_0)$\, .
\end{itemize}
\end{theorem}

\bigskip

\begin{remark}
Intuitively, a Poissonian infinitely divisible process has a $\sigma$-finite L\'evy  measure if and only if there exists $T_0 \in \hT_c$ such that outside of the index set $T_0$ the process  has no nontrivial independent components.
\end{remark}

\medskip

\begin{proposition}\label{c0:sep}
Let $Y=(Y_{t})_{t\in T}$ be a Poissonian infinitely divisible process such that there exists a countable $T_0 \subset  T$ such that every $Y_{t}$ is measurable with respect to $\bar\sigma(Y_{T_0})_{\P}$, the {$\P$-completion} of $\sigma(Y_{t}: t\in T_0)$. Then $Y$ has a $\sigma$-finite L\'evy measure.
\end{proposition}
\np
{\bf Proof.} From the assumption, for every $t \in T$ there exists a Borel measurable function $\Phi_t: \R^{T_0} \mapsto \R$ such that $Y_t = \Phi_t(Y_{T_0})$ a.s. 
Suppose to the contrary that $\nu$ is not $\sigma$-finite, so by Theorem \ref{n-sf} there exist a version $\tilde{Y}$ and  $t_1 \notin T_0$ such that (a)--(c) hold. We also have $\tilde{Y}_{t_1} = \Phi_{t_1}(\tilde{Y}_{T_0})$ a.s.
Hence, for any $u \in \R$
$$
e^{iu(\xi+\eta)} = \E[ e^{iu(\xi+\eta)} \, |\, \widetilde{Y}_{T_0}, \xi] = e^{iu \xi} \E[ e^{iu\eta} \, |\, \widetilde{Y}_{T_0}, \xi] = e^{iu \xi} \E[ e^{iu\eta}]
$$
which gives $|\E[ e^{iu\eta}]|=1$, so that $\eta$ is deterministic. A contradiction. \qed

\begin{definition}\label{sep-P}
	A stochastic process $Y=(Y_t)_{t\in T}$ is said to be separable in probability if there exists  $T_0  \in \hT_c$ such that for any $t \in T$ there is a sequence $\{s_n \} \subset T_0$ such that $Y_{s_n} \cip Y_t$.
\end{definition}

\medskip

From Proposition \ref{c0:sep} we get

\begin{corollary}\label{c:sep}
A separable in probability Poissonian infinitely divisible process $Y$  has a $\sigma$-finite L\'evy measure.
\end{corollary}

\medskip

The fact stated in Corollary \ref{c:sep} was also proved by \citet{Kabluchko-Stoev16} by different methods.  

\medskip

\begin{remark}
{\rm
	Comparing Theorem \ref{t:cLM}(iii) and Corollary \ref{c:sep}, it seem that the $\sigma$-finiteness of a L\'evy measure and the separability in probability of the corresponding Poissonian  process are close. However, the separability in probability is a stronger condition. Indeed, let $V=\{V_t \}_{t\in [0,1]}$ be a  family of i.i.d. Rademacher random variables, $\P(V_t = \pm 1)=1/2$, and let $\nu$ be the probability distribution of $V$ in $(\R^{[0,1]}, \cB^{[0,1]})$. $\nu$ trivially satisfies \eqref{cLM}, so it is a L\'evy measure. Let $Y=\{Y_t \}_{t\in [0,1]}$ be the corresponding  compound Poisson process, see Example \ref{e:CP}. Since $\P(|Y_t-Y_s| > 1) > 1/(2e)$ for any $s\ne t$, $Y$ has finite L\'evy measure but is not separable in probability. 
	} 
\end{remark}

\bigskip

\subsection{Representations and examples of L\'evy measures of processes}\label{s:ex}

A natural way to describe L\'evy measures on path spaces is to view them as ``laws of processes'' defined on possibly infinite measure spaces. Below we  formalize this approach. 

\begin{definition}\label{def:rep}
	Let $\{\nu_I: I \in \hT\}$ be a consistent family of finite dimensional L\'evy measures, which extends uniquely to a L\'evy measure $\nu$ by Theorem \ref{t:LK-P}.  A collection of measurable functions $V=\(V_t \)_{t\in T}$ defined on a measure space $(S, \cS, n)$ is said to be a representation of $\nu$ if  for every $I \in \hT$ 
	\begin{equation} \label{rep}
n(\{s: V_I(s) \in B \}) =  \nu_I(B)\, , \quad \text{for every  } B \in \cB^I_0\, .
\end{equation}
A representation $V$ is called exact if $n \circ V^{-1} = \nu$ or, equivalently, if $n \circ V^{-1}$ satisfies (L\ref{LM0}).  Here $V$ is viewed as a function from $S$ into $\R^T$ given by $V(s)(\cdot)=V_{(\cdot)}$(s). 
\end{definition}

\medskip

Representations of L\'evy measures are useful for stochastic integral and series representations of the corresponding infinitely divisible processes while exact representations give precise forms of L\'evy measures. The difference between these two representations is a technical one, as it is shown below.

\begin{lemma}\label{l:exact}
	Any representation of a L\'evy measure, defined on a $\sigma$-finite measure space, can be modified to an exact representation by restricting it to a smaller domain. 
\end{lemma}

\medskip



\begin{remark}
{\rm
%
 Any L\'evy measure $\nu$ has an exact representation. Indeed, the evaluation process $V_t(x)=x(t)$, $x \in \R^T$, $t\in T$ is an exact  representation of $\nu$ on $(S, \cS, n)=(\R^T, \cB^T, \nu)$. However, such representation does not give much of  information about the L\'evy measure because it is too general. Therefore, we are seeking more specific representations on richer structures, such as standard Borel spaces (Borel subsets of Polish spaces, see \cite[Ch. 1]{Kallenberg(01)}). 
}
\end{remark}

\bigskip

\begin{example}[L\'evy processes]\label{e:Lp}
{\rm 
Let $Y=(Y_t)_{t\ge 0}$ be a Poissonian L\'evy process determined by $\E e^{i u Y_t} = e^{t K(u)}$, where $K$ is the cumulant function given by
$$
K(u) = \int_{\R} (e^{iux} -1 - iu \lt x\rt) \, \rho(dx) + i u c.
$$
For every $I = \{t_1,\dots, t_n\}$, with $0 \le t_1 < \dots < t_n$, and $a=(a_1,\dots,a_n) \in \R^I(\equiv \R^n)$ we have
\begin{equation}\label{cum} 
  \E \exp i \sum_{k = 1}^n a_k Y_{t_k} = \exp \left\{ \sum_{k= 1}^n K(u_k) \Delta t_k   \right\},
\end{equation}
where $\Delta t_k = t_k - t_{k-1}$, $u_k= \sum_{j=k}^{n} a_j$, and $t_0=0$. 
Therefore, the L\'evy measure $\nu_I$ of $X_I$ is given by
\begin{equation} \label{Lm-f}
   \nu_I(B) = \sum_{k=1}^n \int_{\R} \1_{B}(v x_k) \, \rho(dv) \Delta t_k, \quad B \in \mathcal{B}^n
\end{equation}
where $x_k \in \R^n$, $x_k=(\underbrace{0,\dots,0}_{ \ k-1 \, \text{times}}, 1,\dots, 1)$, $k=1,\dots,n$. Define $V=\(V_t \)_{t\in T}$ on  the half-plane  $\R_{+}\times \R$  equipped with a measure $\lambda\otimes \rho$ given by
\begin{equation} \label{rep-Lp}
  V_t(r, v) = \1_{\{t\ge r\}} v, \quad (r,v) \in \R_{+} \times \R,
\end{equation}
where $\lambda$ denotes the Lebesgue measure. We first verify that $V$ is a representation of the L\'evy measure $\nu$ of $Y$. Let $I$ be a finite set of indices as above. For any $B \in \cB^I_0$ we have
\begin{align*} 
  (\lambda\otimes \rho)\{ V_{I} \in B \} &= \int_{\R }\int_{0}^{\infty} \1_{B}(V_I(r,v)) \, dr \rho(dv)\\
  & = \int_{\R } \sum_{k = 1}^n \int_{t_{k-1}}^{t_k} \1_{B}(\1_{\{t_1\ge r\}} v, \dots, \1_{\{t\ge r\}} v) \, dr \rho(dv)\\
  & = \int_{\R } \sum_{k = 1}^n  \1_{B}(\underbrace{0,\dots,0}_{ \ k-1 \, \text{times}}, v,\dots, v) \, \rho(dv) \Delta t_k = \nu_I(B)\, ,
\end{align*}
as in \eqref{Lm-f}. To check that $V$ is exact it is enough to verify \eqref{cLM}. Indeed, for $T_0=\N$ we have
$$
(\lambda\otimes \rho)\{(r,v): \ \1_{\{n\ge r\}} v =0 \ \forall\, n \in \N  \} = 0.
$$
Thus $V$ in \eqref{rep-Lp} is an exact representation of $\nu$.
}
\end{example}

\bigskip

\begin{example}[Squared Bessel processes]\label{e:Bessel}
{\rm 
	Let $Y= \(Y_t \)_{t\ge  0}$ denote a squared Bessel process of dimension $\beta > 0$ starting from 0. If $\beta \in \N$, then $Y_t := \|B_t\|^2$, where $B$ is a $\beta$-dimensional standard Brownian motion.  In general, $Y$ is defined as the unique solution of the stochastic differential equation 
	   $$
   dY_t = 2 \sqrt{Y_t} \, dW_t + \beta\, dt, \quad Y_0=0,
   $$
where $W$ is a one dimensional standard Brownian motion. \citet{Shiga-Watanabe(73)} showed that squared Bessel processes are infinitely divisible and \citet{Pitman-Yor82} described their L\'evy measures on $C(\R_{+})$. We will adapt that characterization to our setting. 

Let 
$$
U_{+} := \{u \in C(\R_{+}): u(0)=0, \ u_{|(0,t_0)}> 0, \ u_{|[t_0, \infty)}=  0 \ \text{for some $t_0>0$} \}. 
$$
$U_{+}$ is a Borel subset of $C(\R_{+})$, on which we consider the It{\^o} measure $n_{+}$ of the Brownian positive excursions, see \citet[Chapter XII]{Revuz-Yor99}. Let $L^a_{\infty}(u)$ denote the total accumulated local time of an excursion $u \in U_{+}$ at $a>0$. Symbolically,
$$
L^a_{\infty}(u) = \int_{0}^{\infty} \delta_a(u(t)) \, dt.
$$
Set $L^a_{\infty}(u) =0$ when $a \le 0$. 
Define $V=\(V_t \)_{t\ge  0}$ on a Borel space $\R_{+}\times U_{+}$  with a measure $\beta \lambda \otimes n_{+}$ by
\begin{equation} \label{rep-Bes}
  V_t(r,u) =  L^{t-r}_{\infty}(u), \quad r \ge 0, \ u \in U_{+}. 
\end{equation}
To check that $V$  is an exact representation of the L\'evy measure $\nu$ of $Y$ we invoke an equation given after Theorem 3.2 in \citet{Mansuy-Yor08} which states that
\begin{equation} \label{MN}
  \int_{U_{+}} \nu(du) F(u) = \beta \int_{U_{+}}  M(du) \int_{0}^{\infty} dr \, F(u((\cdot - r)^{+}))
\end{equation}
for any measurable functional $F: U_{+} \mapsto \R_{+}$, where $M= n_{+} \circ {L^{(\cdot)}_{\infty}}^{-1}$. This equation says that $(\beta \lambda \otimes n_{+}) \circ V^{-1} = \nu$.
	}
\end{example}
\bigskip

\begin{example}[Feller diffusion]\label{e:Feller}
{\rm
We consider a Feller diffusion $Z= \(Z_t \)_{t\ge  0}$ without the drift term, which satisfies the stochastic differential equation
\begin{equation} \label{}
  dZ_t = \sigma \sqrt{Z_t} \, dW_t, \quad Z_0=a>0,
\end{equation}
where $W$ is a one dimensional standard Brownian motion. By change of time, $Y_t= Z_{4\sigma^{-2}t}$ we have
\begin{equation} \label{}
  dY_t = 2\sqrt{Y_t} \, dW'_t, \quad Y_0=a>0,
\end{equation}
where $W'$ is another standard Brownian motion.
Therefore, $Y$ is a 0-dimensional squared Bessel process whose L\'evy measure $\nu_0$ given in the above cited \cite[Theorem 3.2]{Mansuy-Yor08}. Namely, in the notation of  Example \ref{e:Bessel}, $\nu_0 = a \, n_{+} \circ {L^{(\cdot)}_{\infty}}^{-1}$. Therefore, $V=\(V_t \)_{t\ge  0}$ given by
\begin{equation} \label{rep-Fel}
  V_t(u) =  L^{4^{-1}\sigma^2 t}_{\infty}(u), \quad u \in U_{+} 
\end{equation}
is a representation of the L\'evy measure of $Z= \(Z_t \)_{t\ge  0}$ on $(U_{+}, \cB(U_{+}), a \, n_{+})$. 
}
\end{example}

 The next example provides a simple illustration for the method of Lemma \ref{l:exact}. 

\begin{example}[General compound Poisson processes]\label{e:CP}
{\rm  
Let $V=\{V_t\}_{t\in T}$ be a stochastic process and let  $\zeta$  a Poisson random variable with mean $\theta$.  
Let $\{V^{(n)}\}_{n \in \N}$ be a sequence of independent copies of $V$ and  independent of $\zeta$. Then
\begin{equation}
   Y_t= \sum_{n=1}^{\zeta} V^{(n)}_t, \quad  t \in T
\end{equation}
is a Poissonian infinitely divisible process such that  for every $I \in \hat T$,  $a \in \R^I$, 
\begin{equation}\label{ex:CP} 
  \E \exp i \sum_{t \in I} a_t Y_{t} = \exp \left\{ \theta \E (e^{i \sum_{t \in I} a_t V_{t} } - 1)   \right\} = \exp \left\{ \int_{\R^I}(e^{i \langle a, y  \rangle } - 1)\, \nu_I(dy)   \right\}\, .
  \end{equation}
Thus $V=\{V_t\}_{t\in T}$ is a representation the L\'evy measure $\nu$ of $Y$ on $(\Omega, \mathcal{F}, \theta \mathbb{P})$. By the proof of Lemma \ref{l:exact}, the restriction $V_0$ of $V$ to $\Omega_0:= \{\omega: V_{T_0}(\omega) \ne 0 \}$ is an exact representation of $\nu$, where $T_0 \in \hT_c$ is such that 
\begin{equation} \label{rep-CP}
  \P( V_{T_0} = 0) = \inf_{J \in \hT} \, \P\{\omega: V_J(\omega) = 0 \}.
\end{equation}
}
\end{example}

\bigskip

\section{L\'evy-It\^o representations and transfer of regularity for L\'evy measures}\label{s:LI}

\medskip

The following proposition is a direct extension of Proposition \ref{p:can-spe-rep} with a similar proof. Thus its proof will be omitted.

\begin{proposition}\label{p:LI-0}
	Let $X=\(X_t \)_{t \in T}$ be an infinitely divisible process with the generating triplet $(\Sigma, \nu, b)$. Suppose $V=\(V_t \)_{t\in T}$ is a representation of $\nu$ on  $(S, \cS, n)$. Let $G=\(G_t \)_{t \in T}$ be a centered Gaussian process with covariance $\Sigma$ and let $N$ be a Poisson random measure on $(S, \cS)$ with intensity $n$ such that  $G$ and $N$ are independent. Then the process $X' = (X'_{t})_{t\in T}$ given by
	\begin{equation} \label{LI-0}
  X'_t := G_t+\int_{S} V_t(s) \, \big(N(ds) - \chi(V_t(s))\, n(ds)\big) + b(t), \quad t \in T, 
\end{equation}
is a version of the process $X$. 
\end{proposition}

\medskip

The next theorem shows that, under some regularity assumptions, spectral representations  hold almost surely. 
Recall Definition \ref{sep-P} (separability in probability).
\medskip

\begin{theorem}[Generalized L\'evy-It\^o representation]\label{t:LI}
Let $X=\(X_t \)_{t\in T}$ be a separable in probability infinitely divisible process with a separant $T_0$ and the generating triplet $(\Sigma, \nu, b)$. 
Assume that the probability space is rich enough to support independent of $X$ standard uniform random variable. Then, given a representation $V=\(V_t \)_{t\in T}$ of $\nu$ defined on a $\sigma$-finite measure space $(S, \cS, n)$, where $\cS$  is   countably generated (modulo $n$),  there exist a centered Gaussian process $G=\(G_t \)_{t \in T}$ with covariance $\Sigma$, an independent of $G$  Poisson random measure $N$  on $(S, \cS)$ with intensity measure $n$, such that for every $t \in T$ 
\begin{equation} \label{LI}
X_t = G_t + \int_{S} V_t(s) \, \big(N(ds) - \chi(V_t(s))\, n(ds)\big) + b(t) \quad a.s.
\end{equation}
\end{theorem}
 
\bigskip

We illustrate this representation on four examples of infinitely divisible processes.

\begin{corollary}\label{e:LI} { \ }
\np

{\bf (a)}  \underline{\rm L\'evy processes.}
Let $Y=(Y_t)_{t\ge 0}$ be a Poissonian L\'evy process as in Example \ref{e:Lp},  with $\chi= \1{\{|v| \le 1 \}}$. Formula \eqref{LI} applied to $V$ in \eqref{rep-Lp} yields the usual L\'evy-It\^o representation: with probability 1 for all $t\ge 0$,
\begin{align} \label{}
  Y_t & = \int_{\R_{+}}\int_{\R}  \1_{\{t \ge r \}} v \(N(dr, dv) - \chi(\1_{\{t \ge r \}} v)  \, dr \rho(dv) \) + c t \\
  &= \int_{0}^t \int_{|v|\le 1} v   \(N(dr, dv) -  dr \rho(dv)\) +  \int_{0}^t \int_{|v| > 1} v   N(dr, dv)  + c t,
\end{align}
where $N$ is s Poisson random measure on $\R_{+} \times \R$ with intensity $\lambda \otimes \rho$.
\medskip

{\bf (b)} \underline{\rm  Squared Bessel processes.} Let $Y= \(Y_t \)_{t\ge  0}$ be a squared Bessel process of dimension $\beta > 0$ starting from 0, as in Example \ref{e:Bessel}. Formula \eqref{LI} applied to $V$ in \eqref{rep-Bes} yields: with probability 1 for all $t\ge 0$,
\begin{align} \label{}
  Y_t & = \int_{0}^t \int_{U_{+}} L^{t-r}_{\infty}(u) \, N(dr, du) \, ,
\end{align}
where $N$ is s Poisson random measure on $\R_{+} \times U_{+}$ with intensity $\beta \lambda \otimes n_{+}$. Therefore, a squared Bessel process $Y$ is a mixed stochastic convolution. 
\medskip

{\bf (c)} \underline{\rm  Feller diffusion.} Let $Z= \(Z_t \)_{t\ge  0}$ be a Feller diffusion starting from $a>0$, as in Example \ref{e:Feller}. Formula \eqref{LI} applied to $V$ in \eqref{rep-Fel} yields: with probability 1 for all $t\ge 0$,
\begin{align} \label{}
  Z_t  = \int_{U_{+}} L^{\kappa t}_{\infty}(u) \, N(du) \, ,
\end{align}
where $\kappa = \sigma^2/4$ and  $N$ is s Poisson random measure on $U_{+}$ with intensity $a \, n_{+}$.  
\medskip

\underline{\bf (d)}(Compound Poisson process). Let $Y=(Y_t)_{t\in T}$ be a compound Poisson process generated by a stochastic process $V=(V_t)_{t\in T}$ and $\theta>0$, as in Example \ref{e:CP}. Suppose that $V$ is separable in probability. Then \eqref{LI} applied to $V$ yields: with probability 1 for all $t \in T$
\begin{align} \label{}
  Y_t  = \int_{\Omega} V_t(\omega) \, N(d\omega)  \, ,
 \end{align}
	where $N$ is s Poisson random measure on $\Omega$ with intensity $\theta \P$.
\end{corollary}
\bigskip

Next we consider the transfer of regularity for L\'evy measures. In short, this property  says that path regularities of infinitely divisible processes are inherited by representations of their L\'evy measures. 
A precise statement follows.
\medskip

\begin{theorem}[Transfer of regularity]\label{transfer}
Let $X=\(X_t \)_{t\in T}$ be an infinitely divisible process with a $\sigma$-finite L\'evy measure $\nu$. Assume that paths of $X$ lie in a set $U$ that is a standard Borel space for the $\sigma$-algebra $\mathcal{U}= \mathcal{B}^T \cap U$ and also that $U$ is an algebraic subgroup of $\R^T$ under addition. Then $\nu$ is concentrated on $U$ in the sense that $\nu_{\ast}(\R^T \setminus U)=0$. 
Therefore, $\nu$ is well defined on $\mathcal{U}$ and the evaluation process on $(U, \mathcal{U}, \nu)$ is an exact representation of $\nu$ and has paths in $U$.

Moreover, any representation $V=\(V_t \)_{t\in T}$ of $\nu$, given on a $\sigma$-finite measure space $(S, \cS, n)$, can be modified to an exact representation $\widetilde{V}=\(\widetilde{V}_t \)_{t\in T}$ of $\nu$ with  paths in $U$ and such that  $n\{s: \widetilde{V}_t(s)\ne V_t(s) \}=0$ for every  $t \in T$. 
\end{theorem}

\medskip

Let $T=[0,1]$, for concreteness. Examples of $U$ satisfying the above theorem include such ``obvious'' spaces as $C[0,1]$ and $D[0,1]$ (the latter one under the Skorohod topology),  but they  also include non-separable Banach spaces such as the space of Lipschitz-continuous functions $C^{0,1}[0,1]$ and the space of c{\`a}dl{\`a}g functions of finite variation  $BV_1[0,1]$ (which are Borel subsets of $C[0,1]$ and $D[0,1]$, respectively).

\bigskip

As an application, consider a c\`{a}d\`{a}g Poissonian infinitely divisible process $Y=\(Y_t \)_{t\in [0,1]}$. Since $Y$ is right continuous, it is separable in probability, so its L\'evy measure $\nu$ is $\sigma$-finite. By Theorem \ref{transfer},  $\nu$ is concentrated on $D[0, 1]$. Using \citet[Lemma 3.5]{Basse-Rosinski13} we get $\nu\{\|x\|_{\infty} >r \}< \infty$ for any $r>0$. Therefore, the L\'evy-Khintchine representation of Theorem \ref{LK-P} can be refined to
\begin{align} \label{LK-D[0,1]}
    \E & \exp i \sum_{j=1}^n a_j Y_{t_j} \\
 &=  \exp \left\{\int_{D[0,1]} \big(e^{i\sum_{j=1}^n a_j x(t_j) } - 1 - i \sum_{j=1}^n a_j x(t_j) \1_{[0,1]}(\|x\|_{\infty})\big) ~\nu(dx) + i \sum_{j=1}^n a_j b(t_j)  \right\}\, ,
\end{align}
where $b \in D[0,1]$ and $\|\cdot\|_{\infty}$ is the supremum norm.

%
%

\bigskip

\section{Isomorphism identities and spectral representations}\label{s:iso}

We begin with some identities for expectations of functionals of stochastic processes. We present them under different degrees of generality of the assumptions, which is more  suitable for applications. 
For any two $\sigma$-finite measures $\mu$ and $\nu$, we write $\mu \ll \nu$ when $\mu$ is absolutely continuous with respect to $\nu$, and $\mu \sim \nu$ when these measures are equivalent. We will also write $\E[\xi; \, \eta]$ for $\E[\xi\eta]$ and $\E[\xi; \, A] :=\E[\xi \1_A]$, and use the convention $0 \cdot \infty = 0$.



\begin{theorem}\label{t:iso0}
Let  $X=\(X_t \)_{t\in T}$ be an infinitely divisible process having a $\sigma$-finite L\'evy measure $\nu$. Let $Z=\(Z_t \)_{t\in T}$ be a process independent of $X$  such that  $\cL(Z) \ll \nu$. Then \ 
 $\cL(X+Z) \ll \cL(X)$. Hence, there exists a measurable functional $g: \R^T \mapsto \R_{+}$ such that for any measurable functional $F: \R^T \mapsto \R$
\begin{equation} \label{iso0}
  \E F\left(\(X_t+Z_t\)_{t \in T} \right) = \E \left[F\(\(X_t\)_{t \in T}\); \, g(X)\right] \, .
\end{equation}
\end{theorem}
\bigskip

Theorem \ref{t:iso0} is a consequence of Theorem \ref{t:iso1}, which itself is deduced from Theorem \ref{t:iso2}.

\begin{remark}
{\rm 
	(a) From \eqref{iso0} it follows that the linear spaces of functionals $F(X+Z)$ of the process $X+Z$ and $F(X)$ of $X$ are isometric under various norms, such as the $L^p$-norms, but with respect to possibly different probability measures.  This may explain the name  ``isomorphism theorems'' or ``isomorphism identities'' for results of this kind of formulas.  \newline
	(b) The processes $Z$ can be viewed as a random translation (or perturbation) of $X$, so that  Theorem \ref{t:iso0}  gives a sufficient condition when such translation (perturbation) is  ``admissible''.
	}
\end{remark}

The function $g$ in \eqref{iso0} has a closed form only in certain cases. Therefore, we will give below another, easier to handle term in place of $g(X)$. We can also impose a slightly weaker condition on the process $Z$ in part (b), than the one in the previous theorem.

\begin{theorem}\label{t:iso1}
Let  $X=\(X_t \)_{t\in T}$ be an infinitely divisible process of the form $X = G+Y$, where $G=\(G_t \)_{t\in T}$ is a centered Gaussian process independent of a Poissonian process	$Y=\(Y_t \)_{t\in T}$ having a $\sigma$-finite L\'evy measure $\nu$ and given by its canonical spectral representation
\begin{equation} \label{can1}
  Y_t = \int_{\R^T} x(t) [N(dx) - \chi(x(t)) \nu(dx)] + b(t), \quad t \in T.
\end{equation}
Here $N$ is a Poisson random measure with intensity $\nu$ (see Proposition \ref{p:can-spe-rep}). Let $Z=\(Z_t \)_{t\in T}$ be an arbitrary process independent of $N$.

{\bf (a)}
Suppose that $\cL(Z) \ll \nu$ and let $q :=\frac{d\cL(Z)}{d\nu}$ be the Radon-Nikodym derivative of $\cL(Z)$ with respect to $\nu$. Then for any measurable functional $F: \R^T \mapsto \R$ 
\begin{equation} \label{iso1}
  \E F\left(\(X_t+Z_t\)_{t \in T} \right) = \E \left[F\(\(X_t\)_{t \in T}\); \, N(q)\right]
\end{equation} 
where
\begin{equation} \label{Z1}
  N(q) = \int_{\R^T} q(x) \, N(dx)\, .
\end{equation}
Conversely, for any $F$ as above, 
\begin{align} \label{iso1'}
  \E  \left[F\(\(X_t\)_{t \in T}\);  \, N(q) >0\right] =  \E \left[F\left(\(X_t+Z_t\)_{t \in T} \right)\(N(q) + q(Z)\)^{-1}   \right]\, ,
\end{align}
where $q(Z)=q\((Z_{t})_{t\in T}\)$. Moreover, if $\nu\{x: q(x) > 0 \}= \infty$, then $\cL(Y+Z)$ and $\cL(Y)$ are equivalent.

{\bf (b)} 
Suppose that  $\cL(Z) \ll \nu + \delta_{0_T}$ and let $q :=\frac{d\cL(Z)}{d(\nu+ \delta_{0_T})}$. Then for any measurable functional $F: \R^T \mapsto \R$ 
\begin{equation} \label{iso1b}
  \E F\left(\(X_t+Z_t\)_{t \in T} \right) = \E \left[F\(\(X_t\)_{t \in T}\); \, N(q) + q(0_T)\right]\, .
\end{equation} 
Conversely, for any $F$ as above, 
\begin{align} \label{iso1'b}
  \E & \left[F\(\(X_t\)_{t \in T}\);  \, N(q) + q(0_T) >0\right] \\
  & \hspace{0.85in} =  \E \left[F\left(\(X_t+Z_t\)_{t \in T} \right)\(N(q) + q(Z) + q(0_T)\1_{U^c}(Z)\)^{-1}   \right]
\end{align}
where $q(Z)=q\((Z_{t})_{t\in T}\)$ and $U \in \cB^T$ is such that  $0_T \in U \in \cB^T$ and $\nu(U)=0$. Furthermore, $\cL(Y+Z)$ and $\cL(Y)$ are equivalent if $q(0_T) >0$ or  $\nu\{x: q(x) > 0 \}= \infty$.
\end{theorem}

The previous two theorems will be proved as special cases of the next result.

\begin{theorem}\label{t:iso2}
Let  $X=\(X_t \)_{t\in T}$ be an infinitely divisible process given by
\begin{equation} \label{gen-id}
   X_t = G_t + \int_{S} V_t(s) \Big[N(ds) - \chi(V_t(s))  n(ds)\Big] + b(t)\, ,
\end{equation}
where $V=\(V_t \)_{t\in T}$ is a representation of the L\'evy measure of $X$ defined on a $\sigma$-finite measure space $(S, \cS, n)$,  $N$ is a Poisson random measure on $(S, \cS)$ with intensity $n$, $G=\(G_t \)_{t\in T}$ is a centered Gaussian process independent of $N$, and $b$ is a shift function.  Choose an arbitrary measurable function $q: S \mapsto \R_+$ such that  $\int_{S} q(s) \, n(ds)=1$.
 Then  for any measurable functional $F: \R^T \mapsto \R$
\begin{equation}\label{iso2} 
 \int_{S}\E  F\left(\(X_t+ V_t(s)\)_{t\in T} \right) \, q(s)\, n(ds) = 
 \E [F\left( \(X_t\)_{t\in T}\right); \, N(q)]\, , 
\end{equation}
where
\begin{equation} \label{Z2}
  N(q) = \int_{S} q(s) \, N(ds)\, .
\end{equation}

Conversely, for any $F$ as above,
\begin{align} 
   \E  [F\big( & \(X_t\)_{t\in T}  \big)\,  ;  N(q)>0 ] \label{iso3}  \\
   & = 
 \int_{S} \E \big[ F\left(\(X_t+ V_t(s)\)_{t\in T} \right); \, (N(q)+ q(s))^{-1}\big] \,q(s)\, n(ds)\, .
\end{align}
If $n\{s\in S: q(s)>0 \}= \infty$ then
\begin{equation}\label{iso4} 
 \E [ F\left( \(X_t\)_{t\in T}\right) ] = 
 \int_{S} \E \Big[F\big((X_t+ V_t(s)\big)_{t\in T} \big); \,  (N(q)+ q(s))^{-1} \Big] \,q(s)\, n(ds)\, .
\end{equation}  
\end{theorem}




\bigskip

Now we will discuss how Dynkin's isomorphism fits into the pattern of Theorems \ref{t:iso0} and \ref{t:iso1}.

\begin{example}[Dynkin isomorphism for permanental processes]\label{e:Dynkin} { \ }
{\rm 

\noindent
A positive real-valued stochastic process $Y= (Y_x)_{x\in E}$ over a set $E$ is called a $\alpha$-permanental process with kernel $\(u(x,y): x,y \in E \)$ if for every $x_1,\dots, x_n \in E$ and $s_1,\dots, s_n \ge 0$
\begin{equation} \label{per-pr}
  \E \exp\big\{-\sum_{j=1}^{n} s_j Y_{x_j} \big\} = |I + US|^{-\alpha}
\end{equation}
where $U= \(u(x_i,x_j): 1 \le i,j \le n\)$ and $S=\rm{diag}(s_1,\dots,s_n)$ are $n\times n$-matrices, and $\alpha>0$. \\ Hence, $Y_x$'s are gamma distributed with shape parameter $\alpha$ and mean $\alpha u(x,x)$ and jointly they have a multivariate multivariate gamma distribution, as defined by \eqref{per-pr}. A prototype of a permanental process is a squared Gaussian processes, where $u(x,y)$ is the Gaussian covariance multiplied by 2 and $\alpha=1/2$.

For a fixed kernel $\(u(x,y): x,y \in E \)$, let $Y^{(\alpha)}= (Y_x^{(\alpha)})_{x\in E}$ denote the corresponding $\alpha$-permanental process, it it exists. It is easy to see from \eqref{per-pr} that if $Y^{(\alpha)}$ exists and is infinitely divisible for some $\alpha=\alpha_0>0$, then it does exist for every $\alpha$. Conversely,  the existence of $Y^{(\alpha)}$  for every $\alpha>0$ implies that all $Y^{(\alpha)}$ are Poissonian infinitely divisible.

The importance of permanental processes comes also from their connection to Markov processes, as established by \citet{Eisenbaum-Kaspi09}.  They showed in \cite[Theorem 3.1]{Eisenbaum-Kaspi09} that if $X=(X_{t})_{t\ge 0}$ is a transient Markov process with a state space $E$ and 0-potential density $\(u(x,y): x,y \in E \)$ with respect to some reference measure, then  for every $\alpha>0$ there exists a $\alpha$-permanental process~{$Y^{(\alpha)}=(Y_x^{(\alpha)})_{x\in E}$} with kernel $u(x,y)$.
Since $Y^{(\alpha)}= (Y_x^{(\alpha)})_{x\in E}$ must be infinitely divisible and one-dimensional marginals are nonnegative without drift, from Theorem \ref{LK-P} there is a L\'evy measure $\nu$ on $(\R^E, \cB^E)$ such that   
  \begin{equation} \label{per-LK}
  \E \exp\big\{-\sum_{j=1}^{n} s_j Y_{x_j}^{(\alpha)} \big\} = \exp\left[ \int_{\R^{E}_{+}} \left( e^{-\sum_{j=1}^{n} s_j y(x_j)} -1 \right) \, \alpha \nu(dy)\right],
\end{equation}
for all  $x_1,\dots, x_n \in E$, $s_1,\dots, s_n \ge 0$, and $n \ge 1$.  $\nu$ is the L\'evy measure of the \ 1-permanental process. Under some weak assumptions on $Y^{(\alpha)}$, such as its separability in probability, $\nu$ is also $\sigma$-finite, see Corollary \ref{c:sep}. 
 The canonical spectral representation of $Y^{(\alpha)}$ is of the form
\begin{equation} \label{can-perm}
 Y_x^{(\alpha)} = \int_{\R_{+}^E} y(x) \, N^{(\alpha)}(dy), \quad  x\in E
\end{equation}
where $N^{(\alpha)}$ is a Poisson random measure with intensity $\alpha \nu$. 

To formulate the Dynkin Isomorphism Theorem we need more ingredients. Recall,  a transient Markov process  $X=(X_{t})_{t\ge 0}$ specified above. Assume that $X$ admits the local time  $\(L^x_t: x\in E, t \ge 0\)$, which is normalized  to satisfy $\E_x(L^y_{\infty})= u(x,y)$. Fix $a\in E$ with $u(a,a)>0$, and let $\tilde{\P}_a$ be the probability under which the process $X$ starts at $a$ and is killed at its last visit to $a$. Then, for any measurable functional $F: \R^E \mapsto \R$,
\begin{equation} \label{iso-per1}
  \E\tilde{\E}_a \left[ F\((Y^{(\alpha)}_{x} + L^{x}_{\infty})_{x\in E} \) \right] =  \E \left[ F\((Y^{(\alpha)}_{x})_{x\in E} \); \, \frac{Y^{(\alpha)}_a}{\alpha u(a,a)}  \right]\, .
\end{equation}
This identity is a version of the Dynkin Isomorphism Theorem due to  \citet[Theorem 3.2]{Eisenbaum-Kaspi09}. 
Here we assume that the processes $Y^{(\alpha)}=(Y^{(\alpha)}_x)_{x\in E}$ and $L_{\infty}=(L^x_{\infty})_{x\in E}$ depend on different coordinates of the product probability space under the product measure $\P\otimes \tilde{\P}_a$, so that $Y^{(\alpha)}$ and $L_{\infty}$ are independent. 

 To show that \eqref{iso-per1} fits the framework of Theorem \ref{t:iso0}, we will check that $\cL(L_{\infty}) \ll \alpha\nu$. Indeed, by \cite[Lemma 2.6.2]{Marcus-Rosen(book)} and direct computations as in \cite{Eisenbaum-Kaspi09}, we have for every $x_1=a, x_2,\dots, x_n \in E$ and $s_1,\dots, s_n \ge 0$, 
\begin{equation} \label{LT-lt}
  \tilde{\E}_a \exp\big\{-\sum_{j=1}^{n} s_j L^{x_j}_{\infty} \big\} = \frac{1}{u(a,a)}\, \frac{\partial}{\partial s_1} \log |I + U S|\, .
\end{equation}
Combining \eqref{per-pr}, \eqref{per-LK}, and \eqref{LT-lt} we get
\begin{align*} 
  \tilde{\E}_a \exp\big\{-\sum_{j=1}^{n} s_j L^{x_j}_{\infty} \big\} &= -\frac{1}{\alpha u(a,a)}\, \frac{\partial}{\partial s_1} \log |I + U S|^{-\alpha} \\
  &= -\frac{1}{\alpha u(a,a)}\, \frac{\partial}{\partial s_1}\left[ \int_{\R^{E}} \left( e^{-\sum_{j=1}^{n} s_j y(x_j)} -1 \right) \, \alpha \nu(dy)\right] \\
  &= \frac{1}{u(a,a)}\, \int_{\R_{+}^{E}} e^{-\sum_{j=1}^{n} s_j y(x_j)} \, y(a) \, \nu(dy)\, ,
\end{align*}
which implies $\cL(L_{\infty}) \ll \alpha\nu$. 
Now we will deduce \eqref{iso-per1} from Theorem \ref{t:iso1}(a). By the above,   
\begin{equation} \label{}
  q(y) := \frac{d\cL(L_{\infty})}{d(\alpha \nu)}(y) = \frac{y(a)}{\alpha u(a,a)}\, ,  \quad  y \in \R_{+}^E
\end{equation}
 and from  \eqref{can-perm}, 
$$
N^{(\alpha)}(q) = \int_{\R_{+}^E} \, \frac{y(a)}{\alpha u(a,a)} \, N^{(\alpha)}(dy) =\frac{Y^{(\alpha)}_a}{\alpha u(a,a)}\, .
$$
Therefore, \eqref{iso1} gives \eqref{iso-per1}.
Moreover, since $Y^{(\alpha)}_a>0$,  by \eqref{iso1'} we also get 
\begin{equation} \label{iso-per2}
  \E \left[ F\left( (Y^{(\alpha)}_{x})_{x\in E} \right) \right] = 
\alpha u(a,a) \, \tilde{\E}_a \E \left[F\left((Y^{(\alpha)}_x+L^{x}_{\infty})_{x\in E} \right)\(Y^{(\alpha)}_a + L^{a}_{\infty}\)^{-1}  \right]\, .
\end{equation}

Finally, notice that, while  $\cL(L_{\infty})$ is absolutely continuous with respect to $\nu$, it is not equivalent to $\nu$. Indeed, the set
$B=\{y \in \R^E_{+}: y(a)=0 \}$ is a  null set for $\cL(L_{\infty})$ under $\tilde{\P}_a$, but it is not a null set for $\nu$, even in the case when $E$ is a two-point set. Indeed, from \citet{V-J67} we know that the L\'evy measure of a two dimensional permanental vector has singular components on the axes, except for the trivial totally correlated case, so that $\nu(B)>0$ in general. 
Nevertheless, path regularities of a permanental process transfer to $\nu$ by the transfer of regularity for L\'evy measures (Theorem \ref{transfer}), and by the absolute continuity, they transfer to $L_{\infty}=(L^{x}_{\infty})_{x\in E}$.

  }
\end{example}
\bigskip

In Example \ref{e:Dynkin} we have started with an infinitely divisible permanental process  and showed that Dynkin's isomorphism \eqref{iso-per1} is a special case of Theorem \ref{t:iso0}.  N. Eisenbaum \cite[Lemma 3.1]{Eisenbaum08} made a surprising observation that any isomorphism of the type \eqref{iso-per1}  implies the infinite divisibility of the process.  We will reproduce this result in more detail and establish the form of L\'evy measure in a general setting. We begin with random vectors.

\begin{lemma}\label{l:s-iso}
Let $Y=(Y_1,\dots, Y_n)$ be a nonnegative random vector  with $\theta_i=\E(Y_i)\in (0,\infty)$. The following are equivalent:
\vskip8pt
\begin{itemize}
  \item[(i)] $Y$ is infinitely divisible;
  \vskip8pt 
  \item[(ii)] For every $k\le n$ there exists a vector of nonnegative random variables $Z^k=(Z_1^{k},\dots, Z_n^{k})$ independent of $Y$ such that for any bounded measurable functional $F: \R^n \mapsto \R$
\begin{equation} \label{E2}
   \E F(Y+Z^k) = \E [F(Y); \, \theta_k^{-1} Y_k].
\end{equation}
\end{itemize}
Moreover, if (ii) holds, then $Y$ has the L\'evy measure $\nu$ on $\R_{+}^n$ of the form
\begin{equation} \label{R1}
  \nu(dy) = \sum_{k=1}^{n} \theta_k\1_{A_k}(y) y_{k}^{-1} \mathcal{L}(Z^k)(dy)\, ,
\end{equation}
where $A_k=\{y\in \R_{+}^n:  y_{1}=\cdots =y_{k-1}=0,  \ y_{k}>0  \}$.  
The drift of $Y$ equals  
\begin{equation} \label{R2}
  c= (\theta_1 \P(Z^1_1=0),\dots, \theta_n \P(Z^n_n=0)).
\end{equation}
Furthermore, the law of each  $Z^k$ is determined the law of $Y$, $k=1,\dots, n$.
\end{lemma}

\bigskip

Now we characterize processes satisfying the abstract version of Dynkin's Isomorphism.
\begin{proposition}\label{p:s-iso}
	Let $Y=(Y_t)_{t\in T}$ be a nonnegative process with $\theta(t)=\E Y_t<\infty$ for every  $t\in T$.
Suppose that for every $s \in T$ having $\theta(s)>0$, there exists a stochastic process $Z^s=(Z^s_t)_{t\in T}$ independent of $Y$ such that for any measurable functional $F: \R^T \mapsto \R$
	\begin{equation} \label{s-iso}
  \E \big[F((Y_t+Z_t^s)_{t\in T})\big] = \E \big[F((Y)_{t\in T}); \, \theta(s)^{-1} Y_s\big].
\end{equation}
Then the process $Y$ is infinitely divisible. If, in addition, $Y$ is separable in probability with a separant $T_0=(s_k)_{k\ge 1}$, then the L\'evy measure $\nu$ of $Y$ is of the form 
\begin{equation} \label{LM1-s-iso}
  \nu = \sum_{k\ge 1} \1_{A_k}\nu_k\, ,
\end{equation}
where {$A_k=\{y\in \R_{+}^T:  y(s_i)=0 \, \forall\, i < k, \,  y(s_k)>0  \}$} are disjoint and $\nu_k$ are L\'evy measures given by
\begin{equation} \label{LM0-s-iso}
  \nu_k(dy) =  \theta(s_k)\,  y(s_{k})^{-1}\, \mathcal{L}(Z^{s_k})(dy)\, , \quad k=1, 2,\dots .
\end{equation}
The drift of $Y$ is given by $c=\(\theta(t) \P(Z^t_{t}=0): \, t\in T\)$.
 
 Moreover, any nonnegative finite mean infinitely divisible process $Y=(Y_t)_{t\in T}$ satisfies \eqref{s-iso} with 
 $Z^s=(Z^s_t)_{t\in T}$ determined by
 \begin{equation} \label{iso-tr}
   \mathcal{L}(Z^s)(dy) = \frac{c(s)}{\theta(s)} \delta_{0_T}(dy) +  \frac{y(s)}{\theta(s)} \, \nu(dy),
\end{equation}
 where $c$ and $\nu$ are the drift and L\'evy measure of $Y$, respectively, and $\theta(s)=\E Y_s>0$.
\end{proposition}
\bigskip

\begin{remark}\label{}
{\rm
	Consider a separable in probability process $Y$ of Proposition \ref{p:s-iso} with L\'evy measure $\nu$ given by \eqref{LM1-s-iso}. Let $a \in T$ and $\theta(a)>0$. We can always include $a$ in $T_0$ and assume that $s_1=a$. If $Z^a_a>0$ a.s. then $\mathcal{L}(Z^a) \ll \nu_1 \ll \nu$, so that Theorem \ref{t:iso1}{\bf (a)} gives \eqref{s-iso}. If $\P(Z^a_a=0)>0$, then $\mathcal{L}(Z^a) \not\ll \nu $ but $\mathcal{L}(Z^a) \ll \nu + \delta_{0_T}$, so that Theorem \ref{t:iso1}{\bf (b)} gives \eqref{s-iso} in this case.
	}
\end{remark}

\begin{remark}\label{}
{\rm 
If $Y=(Y_1,Y_2)$ satisfies \eqref{E2}, then L\'evy measure of  $Y$ has the form
\begin{align} \label{2-perm}
  \nu(dy) &= \theta_1 \1_{\{y_1>0, y_2=0 \}} y_1^{-1}\mathcal{L}(Z^1)(dy) +   \theta_2 \1_{\{y_1=0, y_2>0 \}} y_2^{-1}\mathcal{L}(Z^2)(dy) \\
  & \ + \theta_1 \1_{\{y_1>0, y_2>0 \}} y_1^{-1}\mathcal{L}(Z^1)(dy).
\end{align}
This formula may shed some light on the form of  L\'evy measure of a 2-dimensional permanental vector in \citet{V-J67}, which has positive masses on the axes.
}
\end{remark}

\bigskip



Isomorphism identities can also be useful for L\'evy processes. We begin with a corollary to  Theorem \ref{t:iso2}.

\begin{corollary}\label{c:Lp}
	Let $X=(X_t)_{t\ge 0}$ be a L\'evy process such that $\E e^{i u X_t} = e^{t K(u)}$, where 
\begin{equation} \label{Lp-LK}
  K(u) = - \frac{1}{2} \sigma^2 u^2 + \int_{\R} (e^{iux} -1 - iu \lt x\rt) \, \rho(dx) + i c u\, . 
\end{equation}
Let $q: \R_{+} \times \R \mapsto \R_{+}$ be a measurable function such that  $\int_{\R_{+} \times \R} q(r,v) \, dr \rho(dv)=1$. Then for any measurable functional $F: \R^{[0,\infty)} \mapsto \R$
\begin{equation}\label{Lp-iso1} 
\E \int_{\R_{+} \times \R} F\left(\(X_t+ \1_{\{r \le t \}} v \)_{t\ge 0} \right) \, q(r,v)\, dr \rho(dv) =  \E [F\left( \(X_t\)_{t\ge 0}\right); \, g(X)]\, , 
\end{equation}
where $g(X) = \sum_{\{r>0: \, \Delta X_r \ne 0 \}} q(r, \Delta X_r)$ and $\Delta X_r = X_r - X_{r-}$. Conversely, 
\begin{align} 
   \E  [F\big( & \(X_t\)_{t\ge 0}  \big)\,  ;  g(X)>0 ] \label{Lp-iso2}  \\
   & = 
 \int_{\R_{+} \times \R} \E \big[ F\left(\(X_t+ \1_{\{r \le t \}} v \)_{t\ge 0} \right); \, (g(X)+ q(r,v))^{-1}\big] \,q(r,v)\, dr \rho(dv)\, .
\end{align}
Moreover, $g(X)>0$ a.s. if $\int_{\R_{+} \times \R} \1\{ q(r,v)>0 \} \, dr \rho(dv) = \infty$.
\end{corollary} 

\np 
{\bf Proof:} This is a direct application of Theorem \ref{t:iso2}. Indeed, has the L\'evy-It\^o decomposition
$$
 X_t = G_t + \int_{\R_{+} \times \R} \1_{\{r \le t \}} v \, \big(N(dr, dv)-  \chi(v) dr \rho(dv) \big)  + c t, 
$$
where $N = \sum_{\{r: \, \Delta X_r \ne 0 \}} \delta_{(r, \Delta X_r)}$. Hence 
$$N(q)= \int_{\R_{+} \times \R} q(r,v) \, N(dr,dv)=\sum_{\{r>0: \, \Delta X_r \ne 0 \}} q(r, \Delta X_r) := g(X), $$ 
as desired. 
\qed
\bigskip

The next example specifies a set of admissible random translations for a Poisson process. 

\begin{example}
	Let $Y=\{Y_t\}_{t\ge 0}$ be a Poisson process with rate $\lambda>0$ and let $Z_t = \1_{[\zeta, \infty)}(t)$, $t \ge 0$, where $\zeta > 0$ is a random variable  with density $h$ and independent of $Y$. Then $\cL(Y+Z) \ll \cL(Y)$ and
$$
 \E [F\left( \(Y_t + Z_t\)_{t\ge 0}\right)] = \E [F\left( \(Y_t\)_{t\ge 0}\right) ; \, g(Y)]
$$
where $g(Y)= \lambda^{-1} \int_{0}^{\infty} h\, dY$. Conversely, 
\begin{equation}
	   \E  [F\big( \(Y_t\)_{t\ge 0}  \big)\,  ;  g(Y)>0 ] = \E  [F\big(  \(Y_t + Z_t\)_{t\ge 0}  \big)\,  ;  (g(Y) + \lambda^{-1} h(\zeta))^{-1} ]\, .
\end{equation} 
Moreover, $g(Y)>0$ a.s. if $\int_{\R_{+}} \1\{ h(r)>0 \} \, dr = \infty$, in which case $\cL(Y+Z)$ and $\cL(Y)$ are equivalent.
\end{example}
\np 
{\bf Proof.} It follows from Corollary \ref{c:Lp} for $q(r, v)=\lambda^{-1}h(r)$ and $\rho=\lambda \delta_1$.
\qed
\medskip

In the previous example function $q$ depended only on time variable. Now we consider the case when $q$ depends only on the space variable. 	

\begin{example}
	Let $X=(X_t)_{t\ge 0}$ be a L\'evy process as in Corollary \ref{c:Lp}. Let $q_1: \R \mapsto \R_{+}$ be such that $\int_{\R} q_1(v) \, \rho(dv) =1$. Consider a fixed time horizon $h>0$, so that $T=[0,h]$ and let
$$
 X_t = G_t + \int_{[0,h] \times \R} \1_{\{r \le t \}} v \, \big(N(dr, dv)-  \chi(v) dr \rho(dv) \big)  + c t, \quad t \in [0,h]\, .
$$
To apply Theorem \ref{t:iso2}, take $q(r,v) = h^{-1} q_1(v)$, so that $\int_{[0,h] \times \R} q(r,v) \, dr \rho(dv)=1$, and compute
\begin{align*} 
  N(q) = \int_{[0,h] \times \R} q(r,v) \, N(dr, dv) = h^{-1} \int_{[0,h] \times \R} q_1(v)  \, N(dr, dv) := h^{-1} W_h\, ,
\end{align*}
where $W=(W_{h})_{h\ge 0}$ is a subordinator with L\'evy measure $\rho_1$ given by $\rho_1(x,\infty) = \rho(q_1^{-1}(x,\infty))$, $x\ge 0$. Then, for every  measurable $F: \R^{[0,h]} \mapsto \R$
\begin{equation} \label{Lp-iso3}
 \E \int_{\R} F\left(\(X_t+ \1_{\{r \le t \}} v \)_{t\in [0,h]} \right) \, q_1(v)\, \rho(dv) =  h^{-1} \E [F\left( \(X_t\)_{t\in [0,h]}\right); \, W_h]\, , 
\end{equation}
and
\begin{align} 
   \E  [F\big( & \(X_t\)_{t\in [0,h]}  \big)\,  ;  W_h>0 ] \label{Lp-iso4}  \\
   & = 
h \int_{\R} \E  \big[ F\left(\(X_t+ \1_{\{r \le t \}} v \)_{t\in [0,h]} \right); \, ( W_h+ q_1(v))^{-1}\big] \, q_1(v)\,  \rho(dv)\, .
\end{align}
\end{example}

\bigskip

To illustrate usefulness of these formulas, in the next corollary we  give an alternative proof to a known fact on the behavior of the distributions of the L\'evy process at the origin. Actually, we prove a more general version of this fact, see, e.g., \cite[Corollary 8.9]{Sato(99)}, and the last statement of the corollary seems to be new. 
 
\begin{proposition}\label{p:Lp0}
	Let $X=(X_t)_{t\ge 0}$ be a L\'evy process and let $\rho$ be the L\'evy measure of $X_1$. 
Let $f: \R   \mapsto \R$ be a bounded function continuous on a set of full $\rho$-measure  such that $f(x)=o(x^2)$ as $x \to 0$ (or $f(x)=O(x^2)$ when $X$ has no Gaussian component).  Then
\begin{equation} \label{LP3}
  \lim_{h\to 0} h^{-1} \E f(X_h) =  \int_{\R } f(v) \, \rho(dv)\, .
\end{equation}
If $X$ is a subordinator, then the assumption of continuity can be weakened to the right-continuity.
\end{proposition}

\medskip

\begin{corollary}\label{}
	If $X$ is a  L\'evy process then for any $\rho$-continuity set $B$ with $0 \notin \bar{B}$ we have $\lim_{h\to 0} h^{-1} \P(X_{h} \in B) = \nu(B)$. If in addition $X$ is a subordinator, then for every $r > 0$ we  have ${\lim_{h\to 0} h^{-1} \P(X_{h} \in [r, \infty)) = \nu([r, \infty))}$.
\end{corollary}
\np 
{\bf Proof:} In the first part of the corollary we take $f(x)= \1_{B}(x)$ and in the second part we take $f(x) = \1{[r, \infty) }(x)$, and apply the above.

\bigskip

\section{Series representations and isomorphism identities}\label{s:series}

Here we will show how representations of L\'evy measures lead to series representations of Poissonian infinitely divisible processes. This method of constructing series representations was initiated in \cite{Rosinski90} and further developed in \cite{Rosinski01}.

\bigskip

\begin{theorem}\label{t:series}
Let $Y=\(Y_t \)_{t\in T}$ be a Poissonian infinitely divisible process with the generating triplet $(0, \nu, b)$. Let $V=\(V_t \)_{t\in T}$ be a representation of $\nu$ on a $\sigma$-finite measure space $(S, \mathcal{S}, n)$. Consider a probability measure $n^{(1)}$  equivalent to $n$, so that $n^{(1)}(ds) = g(s) n(ds)$ for some $g>0$ $n$-a.e.  Let $\(\xi _j \)_{j\in \N}$ be an i.i.d. sequence of random elements in $S$ with the common distribution $n^{(1)}$ 
 and let $\(\Gamma_j \)_{j\in \N}$ be a sequence of partial sums of i.i.d. mean-one exponential random variables independent of the sequence $\(\xi_j \)_{j\in \N}$.  
Then 
\begin{itemize}
  \item[\rm{(i)}] For every $t \in T$ the series
  \begin{equation} \label{S0}
  S_t^{(0)}:=\sum_{j=1}^{\infty} V_t(\xi_j)\1{\{g(\xi_j) \le \Gamma_j^{-1}\}}
\end{equation}
converges a.s. if and only if  the limit
\begin{equation} \label{S-C}
  c(t) := \lim_{j \to \infty} \int_{S} \lt V_t(s)\rt \, (j g(s) \wedge 1) \, n(ds) \quad  \text{exists}\, .
\end{equation}
If \eqref{S-C} holds then 
\begin{equation} \label{S0-rep}
 \(Y_t \)_{t\in T} \eid \(S_t^{(0)} + b(t) - c(t) \)_{t\in T}\, .
\end{equation}  
 In particular, if $Y_t \ge 0$, or more generally, if $\int_{\R^T} |x(t)| \wedge 1 \, \nu(dx) < \infty$, then \eqref{S-C} holds. 
 \smallskip
 \item[\rm{(ii)}] For every $t \in T$ the centered series
   \begin{equation} \label{S1}
  S_t:=\sum_{j=1}^{\infty} \Big[V_t(\xi_j)\1{\{g(\xi_j) \le  \Gamma_j^{-1}\}} - c_j(t)\Big]
\end{equation}
converges a.s. and 
\begin{equation} \label{S1-rep}
 \(Y_t \)_{t\in T} \eid \(S_t + b(t) \)_{t\in T}\, ,
\end{equation}
where
\begin{equation} \label{cj}
   c_j(t) = \int_{S} \lt V_t(s)\rt \, \big\{(j g(s) \wedge 1) - ((j-1) g(s) \wedge 1) \big\} \, n(ds)\, .
\end{equation}
Under \eqref{S-C}, $\sum_{j=1}^{\infty} c_j(t) = c(t)$, so that $S_t = S^{(0)}_t - c(t)$.
\end{itemize}
\end{theorem}
\begin{proof}
	This proof is a routine application of Theorem 4.1 \cite{Rosinski01} and thus it is omitted.  
\end{proof}
 
\bigskip

\begin{remark}\label{r:series}
{\rm 
{\bf (a)} If $Y$ is separable in probability and defined on a rich enough probability space (see Theorem \ref{t:LI}), then proceeding as in \cite{Rosinski01},  we can choose $V_j, \Gamma_j$ on the same probability space as $Y$ such that \eqref{S0-rep} and \eqref{S1-rep} hold not only in distribution but also almost surely. 
\smallskip

\noindent{\bf (b)} If sample paths of $Y$ belong to a separable Banach space, then the pointwise convergent series \eqref{S0} and \eqref{S1} converge a.s.  in the norm of that Banach space. Such conclusion is generally false when sample paths of $Y$ belong to a non-separable Banach space. An exception is the Skorohod space under the uniform topology, which is not separable. However, in such space the series converge uniformly a.s., see  \cite{Basse-Rosinski13}. 
\smallskip

\noindent{\bf (c)} There is some analogy between series expansions of  Poissonian infinitely divisible process, such as in Theorem \ref{t:series}, and Karhunen-Lo\`{e}ve series representation of Gaussian processes. Exploring this analogy, one has the corresponding results for the oscillation and zero-one laws of Poissonian infinitely divisible process. 	See \cite{Cambanis-Nolan-Rosinski90} and \cite{Rosinski90a}.
}
\end{remark}

%
%

\begin{example}[Feller diffusions]\label{e:series-Fel} 
{\rm Let $Z=(Z_{t})_{t\in T}$ be a Feller diffusion, as in Example \ref{e:Feller}. Recall that $V_t = L^{\kappa t}_{\infty}$, $t \ge 0$ is a representation of the L\'evy measure of $Z$ on 
$(S, n)= (U_{+}, a \, n_{+})$. We will now give a probability measure $n^{(1)}$ equivalent to $a \, n_{+}$. Let $R(u)$ denote the length of an excursion $u \in U_{+}$. It is well-known that
$$
n_{+}\{u: R(u) > x \} = \frac{1}{\sqrt{2\pi}} x^{-1/2}
$$
see, e.g., \cite[Ch. 12, Proposition 2.8]{Revuz-Yor99}. Let $f: \R_{+} \mapsto \R_{+}$ be such that $f(x)=0$ only for $x=0$ and 
$$
\frac{1}{\sqrt{2\pi}} \int_{0}^{\infty} f'(x) x^{-1/2} \, dx = 1. 
$$
Put $n^{(1)}(du) := f(R(u))\, n_{+}(du)$. Then
$$
\int_{U_{+}} f(R(u))\, n_{+}(du) = \int_{0}^{\infty} f'(x) \, n_{+}\{u: R(u) > x \} \, dx = 1,
$$
so that $n^{(1)}$ is a probability measure such that $g(u) :=\frac{dn^{(1)}}{d(a \, n_{+})}(u) = a^{-1} f(R(u))$. Now we will apply Theorem \ref{t:series}(a). Let $\(\xi _j \)_{j\in \N}$ be an i.i.d. sequence of random elements in $U_{+}$ with the common distribution $f(R)\, dn_{+}$
and let $\(\Gamma_j \)_{j\in \N}$ be a sequence of partial sums of i.i.d. mean-one exponential random variables independent of the sequence $\(\xi_j \)_{j\in \N}$. Then
\begin{equation} \label{}
 Z_t \eid \sum_{j=1}^{\infty} L^{\kappa t}_{\infty}(\xi_j)\1{\{f(R(\xi_j)) \le  a\Gamma_j^{-1}\}}, \quad t \ge 0
\end{equation}
in the sense of equality of finite dimensional distributions. By Remark \ref{r:series}(b), the convergence holds also a.s. uniformly in $t$ on finite intervals. Let us take $f(x)=\sqrt{\frac{\pi}{2}}(x \wedge 1)$ for concreteness. Then the above formula becomes
$$	
 Z_t \eid   \sum_{j=1}^{\infty} L^{\kappa t}_{\infty}(\xi_j)\1\{ R(\xi_j) \wedge 1 \le  (2/\pi)^{1/2} a \Gamma_j^{-1}\}, \quad t \ge 0.
$$
This formula says that a Feller diffusion is the series of randomly trimmed total accumulated local times taken at the level $\kappa t$, $t \ge 0$ from an infinite sample of Brownian excursions. This sample is taken according to the density $(\pi/2)^{1/2}(R \wedge 1)$ with respect to $n_{+}$. 
}	
\end{example}

Along similar lines we obtain series representations of squared Bessel process.

\begin{example}[Squared Bessel processes]\label{e:series-Bes} 
{\rm 
Let $Y=(Y_{t})_{t\in T}$ be a squared $\beta$-dimensional Bessel process starting from 0 and $\beta>0$, as in Example \ref{e:Bessel}. Recall that $V_t = L^{t-(\cdot)}_{\infty}$, $t \ge 0$ is a representation of the L\'evy measure of $Y$ on $(S, n)= (\R_{+}\times U_{+}, \beta \lambda \otimes n_{+})$. Recall that $L^a_{\infty}(u) =0$ when $a \le 0$.
  Let $f$ and $R$ be as in Example \ref{e:series-Fel}. Put
 $$
n^{(1)}(dr, du) = (\beta e^{-\beta r} dr) \otimes f(R(u)) n_{+}(du), \quad r>0, \ u \in U_{+}. 
$$
$n^{(1)}$ is a probability measure  equivalent to $\beta \lambda \otimes n_{+}$. Let $\{\eta_{n}\}$ be an i.i.d. sequence of exponential random variables with parameter $\beta$, let $\(\xi _j \)_{j\in \N}$ be an i.i.d. sequence of random elements in $U_{+}$ with the common distribution $f(R)\, dn_{+}$
and let $\(\Gamma_j \)_{j\in \N}$ be a sequence of partial sums of i.i.d. mean-one exponential random variables. Assume that these sequences are independent of each other. We compute 
$$
g(r, u) = \frac{dn^{(1)}}{dn}(dr, du) = e^{-\beta r} f(R(u)).
$$
By Theorem \ref{t:series}(i) and Remark \ref{r:series}(b), 
$$
Y_t \eid  \sum_{j=1}^{\infty} L^{t-\eta_j}_{\infty}(\xi_j)\1{\{e^{-\beta \eta_j}f(R(\xi_j)) \le \Gamma_j^{-1}\}}, \quad t \ge 0\, 
$$
in the sense of equality of finite dimensional distributions and the series converges uniformly a.s. 
Again, choosing a specific $f$,  as at the end of Example \ref{e:series-Fel}, may give more insight into this representation.
}	
\end{example}

Series representations of L\'evy processes have been considered in many places, 
so we will only sketch representations resulting from Theorem \ref{t:series}.

\begin{example}[L\'evy process]\label{e:series-Lp} 
{\rm 
Let $Y=(Y_t)_{t\ge 0}$ be a Poissonian L\'evy process, as in Example \ref{e:Lp}. Then $V_t(r,v) = \1_{\{r \le t \}} v$ is a representation of the L\'evy measure on $(\R_{+} \times \R, \lambda \otimes \rho)$. Choose an arbitrary probability measure $n^{(1)}$ on $\R_{+} \times \R$ that is equivalent to $\lambda \otimes \rho$ and put $g(r,v)= \frac{dn^{(1)}}{d(\lambda \otimes \rho)}(r,v)$. Let $\xi_j=(\eta_j, \upsilon_j)$, $j \in \N$ be i.i.d. random variables with the common density $g$ with respect to $\lambda \otimes \rho$. Then by Theorem \ref{t:series}(ii)  and Remark \ref{r:series}(b), 
\begin{equation} \label{}
  Y_t \eid \sum_{j=1}^{\infty} \Big[\1_{\{\eta_j \le t \}} \upsilon_j \1{\{g(\eta_j, \upsilon_j) \le \Gamma_j^{-1}\}} - c_t(j)\Big] + b_t
\end{equation}
in the sense of equality of finite dimensional distributions  and the series converges uniformly almost surely. 
}	
\end{example}

\bigskip


\begin{theorem}[Series form of isomorphism]\label{t:iso-s}
	Under notation of Theorem \ref{t:series}, consider a Poissonian infinitely divisible process $Y=\(Y_t \)_{t\in T}$ having L\'evy measure $\nu$ and given  by \eqref{S1}	
\begin{equation} \label{}
  Y_t = \sum_{j=1}^{\infty} \Big[V_t(\xi_j)\1{\{g(\xi_j) \le  \Gamma_j^{-1}\}} - c_j(t)\Big] + b(t)\, ,
\end{equation}
where $V=\(V_t \)_{t\in T}$ is a representation of $\nu$ on  $(S, \mathcal{S}, n)$ and $g = \frac{d \mathcal{L}(\xi_1)}{d \nu} >0$.
Let $\xi_0$ be a random variable in $S$ independent of $\(\xi_j, \Gamma_j\)_{j\in \N}$ such that $\cL(\xi_0) \ll \cL(\xi_1)$ and let $q= \frac{d \mathcal{L}(\xi_o)}{d \nu}$. Then for any measurable functional $F: \R^T \mapsto  \R$
\begin{equation} \label{iso-s1}
  \E \left[ F\( (V_t(\xi_0) + Y_t)_{t \in T} \) \right] = \E \left[ F\( (Y_t)_{t \in T}\); \, Q \right] 
\end{equation}
where 
\begin{equation} \label{Q}
  Q = \sum_{j=1}^{\infty} q(\xi_j)\1\{g(\xi_j) \le \Gamma_j^{-1}\}\, .
\end{equation}
Conversely,
\begin{equation} \label{iso-s2}
  \E \left[ F\( (Y_t)_{t \in T}\); \, Q >0\right] = \E \int_{S}  \big[ F\left(\(V_t(s)+ Y_t\)_{t\in T} \right); \, (Q+ q(s))^{-1}\big] \, q(s)\,  n(ds)  \, .
\end{equation}
Moreover $Q>0$ a.s. provided $n\{s: q(s)>0 \}= \infty$.
\end{theorem}

\bigskip

\section{Proofs}\label{s:Proofs}

\np 
{\bf Proof of Lemma \ref{l:nu*}. } Clearly, the left hand side is greater or equal than the right hand side  in \eqref{nu*}. To prove the reverse inequality,  take $A \in \cB^T$. Since
$$
 \nu_{*}(A \setminus 0_T) = \sup\{ \nu(D): \ D \subset A \setminus 0_T, \ D \in \cB^T\} \, ,
$$
there exists an
	$A_0 \in \cB^T$  with $A_0 \subset A \setminus 0_T $  such that $\nu(A_0) = \nu_{*}(A \setminus 0_T)$. From the structure of $\cB^T$, $A_0 = \pi_U^{-1}(B)$ for some set $U \in \hT_c$ and $B \in \cB^U$, with $0_U \notin B$. Let $J_n \in \hT$ be such that $J_n \uparrow U$. By the continuity of $\nu$ from below, 
$$
\nu_{*}(A \setminus 0_T)=\nu(A_0)= \nu(A_0 \setminus  \pi_U^{-1}(0_U)) = \lim_{n\to \infty} \nu(A_0 \setminus  \pi_{J_n}^{-1}(0_{J_n})) \le \sup_{J \in \hT} \nu(A \setminus \pi^{-1}_J(0_J) )\, ,
$$
which establishes \eqref{nu*}. 

Since the measures $\nu( \cdot \setminus \pi^{-1}_J(0_J) )$, $J \in \hT$ are increasing as $J$'s are increasing and $\hT$ is a directed set under the inclusion, 
$$
\nu^0(\cdot) = \sup_{J \in \hT} \nu( \cdot \setminus \pi^{-1}_J(0_J) )
$$
is a measure. If $0_{T} \notin A \in  \cB^T$, then
$$
\nu^0(A) = \nu_{*}(A \setminus 0_T) = \nu_{*}(A) = \nu(A),
$$
which gives \eqref{nu_0}. Using \eqref{nu*}-\eqref{nu_0} we get for every $A \in  \cB^T$ 
$$
(\nu^0)_{\ast}(A \setminus 0_T) = \sup_{J \in \hT} \nu^0(A \setminus \pi^{-1}_J(0_J) ) = \sup_{J \in \hT} \nu(A \setminus \pi^{-1}_J(0_J) ) = \nu^0(A),
$$
which establishes (L\ref{LM0}).
Finally, if $\nu$ satisfies (L\ref{LM}), then so does $\nu^0$. \qed

\bigskip

\np
{\bf  Proof of Lemma \ref{LM-equiv}. } \  (L\ref{LM0}) $\Rightarrow$ (a). Let $T_0 \in \hT_c$. By (L\ref{LM0}) and \eqref{nu*}  there exists an increasing sequence  $J_n \in \hT$ such that 
$$
\nu\{x: x_{T_{0}} = 0\}= \lim_{n \to \infty} \nu\{x: x_{T_{0}} = 0, \ x_{J_{n}} \ne 0\}.
$$
Therefore, $T_1= \bigcup J_n$ satisfies (a).
\smallskip

\np
(a) $\Rightarrow$ (b). Let $J_n$ be as above. (a) implies that for some $n$,  $\nu\{x: x_{T_{0}} = 0, \ x_{J_{n}} \ne 0\}>0$. Hence \eqref{unc1} holds for some $t \in J_n$.
\smallskip

\np 
(b) $\Rightarrow$ (c). If \eqref{cLM} and \eqref{unc1} do not hold, then we have a contradiction with (b). 
\smallskip

\np 
(c)  $\Rightarrow$ (L\ref{LM0}). Since by Remark \ref{r:cLM}, \eqref{cLM} implies (L\ref{LM0}), we only need to consider the second part of the alternative in (c); i.e.,  we assume that \eqref{unc1} holds for any set $T_0 \in \hT_c$. 
Let  $A \in \cB^T$. (L\ref{LM0}) obviously holds when $0_T \notin A$, so we consider the case $O_T \in A$. Using \eqref{nu*}, as in the first implication of this proof, 
we infer that $\nu_{\ast}(A \setminus 0_T) = \nu(A \setminus \{x: x_{T_1} = 0 \})$ for some $T_1 \in \hT_c$. There is also a countable set $T_0 \supset T_1$ and $B \in \mathcal{B}^{T_0}$ such that $A = \{x: x_{T_0} \in B \}$. Let $t \notin T_0$ be such that $\nu\{x: x_{T_0} = 0, \ x(t) \ne 0\} = \alpha > 0$. Since $\nu_{\ast}(A \setminus 0_T) = \nu(A \setminus \{x: x_{T_1} = 0 \})$ holds also for any larger set in place of $T_1$, we get
\begin{align*} 
  \nu_{\ast}(A \setminus 0_T) &= \nu(A \setminus \{x: x_{T_0 \cup \{t \}} = 0 \}) \\
  &= \nu(A \setminus \{x: x_{T_0} = 0 \}) + \nu(A \cap \{x: x_{T_0} = 0 \} \cap \{x: x(t) \ne 0 \}) \\
  &= \nu(A \setminus \{x: x_{T_0} = 0 \}) + \nu(\{x: x_{T_0} = 0, \ x(t) \ne 0 \}) \\
  & = \nu_{\ast}(A \setminus 0_T) + \alpha
\end{align*}
where in the third equality we used that $0_T \in A$. The above computation shows that $\nu_{\ast}(A \setminus 0_T) = \infty$, in which case (L\ref{LM0})  trivially holds. The proof is complete. 
\qed

\bigskip

\np
{\bf  Proof of Theorem \ref{t:LK-P}.}   
The proof of Theorem \ref{LK-P} is divided into four steps.

\medskip

{\it Step 1. 
Let $J \in \hat T$ and $\epsilon>0$  be fixed. Set $U=\{y\in \R^J: \max_{t\in J} |y_t|>\epsilon \}$ and define a family of measures \ $\lambda_{I}^{J, \epsilon}$ on $(\R^I, \mathcal{B}^I)$ by
\begin{equation} \label{lje}
  \lambda_{I}^{J, \epsilon}(B) = \nu_K\left(\pi^{-1}_{KI}(B) \cap  \pi^{-1}_{KJ}(U) \right),  \quad B \in 
\mathcal{B}^I,
\end{equation}
where $I \in \hT$ and $K = I \cup J$. We will show that there exists a finite measure $\lambda^{J, \epsilon}$ on $(\R^T, \cB^T)$ such that 
\begin{equation} \label{step1}
  \lambda^{J, \epsilon} \circ \pi^{-1}_{I} =  \lambda_{I}^{J, \epsilon}  \quad \text{for all } I \in \hat{T}.
\end{equation}
}

First observe that all measures $\lambda_{I}^{J, \epsilon}$, $I \in \hat T$, have equal finite mass. Indeed,  by (c\ref{c3}) and the fact that $\nu_J$ is a L\'evy measure on $\R^J$
$$
\lambda_{I}^{J, \epsilon}(\R^I) = \nu_K\left(\pi^{-1}_{KJ}(U) \right) =\nu_J(U) < \infty. 
$$
Next we will show that $\{\lambda^{J, \epsilon}_I: I \in \hat T \}$ is a projective system. Take $I_1 \subset I_2 \in \hT$ and put $K_1=J \cup I_1$, $K_2=J \cup I_2$. We have for every $B \in \cB^{I_1}$
\begin{align*}
        \lambda_{I_2}^{J, \epsilon} \left( \pi_{I_2,I_1}^{-1}(B) \right) &=   
        \nu_{K_2}\left(\pi^{-1}_{K_2,I_2}(\pi_{I_2,I_1}^{-1}(B)) \cap  \pi^{-1}_{K_2,J}(U) \right) \\
         &=   
        \nu_{K_2}\left(\pi^{-1}_{K_2,I_1}(B) \cap  \pi^{-1}_{K_2,J}(U) \right) \\
        &= \nu_{K_2}\left(\pi^{-1}_{K_2,K_1}  \left( \pi_{K_1,I_1}^{-1}(B) \cap  \pi^{-1}_{K_1,J}(U) \right) \right) \\
        &= \nu_{K_1}\left( \pi_{K_1,I_1}^{-1}(B) \cap  \pi^{-1}_{K_1,J}(U) \right) = \lambda_{I_1}^{J, \epsilon}(B). 
\end{align*}
In the fourth equality we  used (c\ref{c3}) as $ \pi_{K_1,I_1}^{-1}(B) \cap  \pi^{-1}_{K_1,J}(U)$ does not contain the origin of $\R^{K_1}$. By Kolmogorov's Extension Theorem there exists a finite measure $\lambda^{J, \epsilon}$ on $(\R^T, \cB^T)$ satisfying \eqref{step1}. 

\bigskip

{\it Step 2. Define
\begin{equation} \label{step2}
  \nu(A) = \, \sup_{J \in \hT, \, \epsilon>0} \,  \lambda^{J, \epsilon}(A), \quad A \in \cB^{T}.
\end{equation}
Then $\nu$ is a measure satisfying \eqref{LM-cons}. 
}
\medskip

First we observe from \eqref{lje} that for every $I \in \hT$, $\lambda_{I}^{J_1, \epsilon_1} \le \lambda_{I}^{J_2, \epsilon_2}$ whenever $J_1 \subset J_2$ and $\epsilon_1 \ge \epsilon_2$. This implies, in conjunction with \eqref{step1}, that $\lambda^{J_1, \epsilon_1}(A) \le \lambda^{J_2, \epsilon_2}(A)$ for every $A$ from the algebra of cylinders, $A \in \pi^{-1}_I(\cB^I)$, $I \in \hT$. By the monotone class argument we obtain
\begin{equation} \label{lje-m}
  \lambda^{J_1, \epsilon_1} \le \lambda^{J_2, \epsilon_2} \quad \text{when } J_1 \subset J_2 \ \text{and } \epsilon_1 \ge \epsilon_2.
\end{equation}
We will now check that $\nu$ is a measure. For any pairwise disjoint sets $A_n \in \cB^T$ we have
\begin{align*} 
  \nu\left( \bigcup_{i=1}^{\infty} A_i \right) &= \, \sup_{J \in \hT, \, \epsilon>0} \,  \lambda^{J, \epsilon}\left( \bigcup_{i=1}^{\infty} A_i \right) = \, \sup_{J \in \hT, \, \epsilon>0} \, \sup_{n \in \N} \, \lambda^{J, \epsilon}\left( \bigcup_{i=1}^{n} A_i \right) \\
  &=  \,  \sup_{n \in \N} \, \sup_{J \in \hT, \, \epsilon>0} \, \lambda^{J, \epsilon}\left( \bigcup_{i=1}^{n} A_i \right) = \,  \sup_{n \in \N} \, \sup_{J \in \hT, \, \epsilon>0} \, \sum_{i=1}^{n}  \lambda^{J, \epsilon}\left( A_i \right) \\
  &= \,  \sup_{n \in \N} \,  \sum_{i=1}^{n} \, \sup_{J \in \hT, \, \epsilon>0} \, \lambda^{J, \epsilon}\left( A_i \right) =  \sum_{i=1}^{\infty}  \nu\left(A_i \right),
\end{align*}
where the fifth equality uses \eqref{lje-m}. Now we will show that  measure $\nu$ satisfies  \eqref{LM-cons}.  Let $I \in \hT$ and $B \in \mathcal{B}(\R^I)$. We have
\begin{align*} 
   \nu(\pi^{-1}_{I}(B \setminus 0_I)) &  \ge \, \sup_{\epsilon>0} \,  \lambda^{I, \epsilon}(\pi^{-1}_{I}(B \setminus 0_I)) \\
   & = \, \sup_{\epsilon>0} \, \nu_I(B \cap \{y\in \R^I: \max_{t\in I} |y_t|>\epsilon \})  = \nu_I(B),
\end{align*}
and, conversely,  
$$
\nu_I(B)= \nu_I(B \setminus 0_I) \ge \lambda^{J, \epsilon}_I(B \setminus 0_I) = \lambda^{J, \epsilon}(\pi^{-1}_{I}(B \setminus 0_I))
$$   
for every $J \in \hT$. Taking supremum over $(J, \epsilon)$  shows \eqref{LM-cons}, which completes Step 2.

\bigskip

{\it Step 3. 
$\nu$ is a L\'evy measure.
}
\medskip

We need to show (L\ref{LM0}). In view of \eqref{nu*} it is enough to show that
\begin{equation} \label{nu0}
 \nu(A)= \sup_{J \in \hT} \nu(A \setminus \pi^{-1}_J(0_J) ) \quad \text{for every } A \in \cB^T.
\end{equation}

First we will show that for any $J \in \hT$ and $\epsilon>0$
\begin{equation} \label{l0}
 \sup_{H \in \hT}  \lambda^{J, \epsilon}(A \setminus \pi^{-1}_H(0_H)) =  \lambda^{J, \epsilon}(A), \quad \text{for every } A \in \cB^T.
\end{equation}
By an argument similar to the proof of countable additivity of $\nu$ in the previous step, we infer that $A \mapsto \sup_{H \in \hT}  \lambda^{J, \epsilon}(A \setminus \pi^{-1}_H(0_H))$ is a finite measure on $\cB^T$.  Therefore, it is enough to show the equality in \eqref{l0} on the algebra of cylinders. Since ``$\le$'' is obvious, we will prove the opposite inequality.  Let $A= \pi^{-1}_I(B)$, where $I \in \hT$ and $B \in \cB^I$, and let $K=I\cup J$. We have
\begin{align*} 
  \sup_{H \in \hT}  \lambda^{J, \epsilon}(A \setminus \pi^{-1}_H(0_H)) &\ge \lambda^{J, \epsilon}(A \setminus \pi^{-1}_{K}(0_{K})) = \lambda^{J, \epsilon}(\pi^{-1}_K(\pi^{-1}_{KI}(B) \setminus 0_K)) \\
  &= \lambda^{J, \epsilon}_K(\pi^{-1}_{KI}(B) \setminus 0_K) = \lambda^{J, \epsilon}_K(\pi^{-1}_{KI}(B))\\
  &= \lambda^{J, \epsilon}(\pi^{-1}_{K}(\pi^{-1}_{KI}(B))) = \lambda^{J, \epsilon}(A)
\end{align*}
The second and the fifth equations use \eqref{step1}, and the third one follows from the definition of $\lambda^{J, \epsilon}_K$ and that $0_K \in \{y\in \R^K: \max_{t\in J} |y_t| \le \epsilon \}$. This proves \eqref{l0}. Taking supremum over $(J, \epsilon)$ in \eqref{l0} yields \eqref{nu0}.

\bigskip

{\it Step 4. $\nu$ is the smallest measure satisfying \eqref{LM-cons}, so  is unique.} 
\medskip

Suppose that $\rho$ also satisfies \eqref{LM-cons}. 
Let $A=\pi^{-1}_{I}(B)$,  where $I \in \hat T$ and $B \in \cB^I$. Using 
(L\ref{LM0}) and Lemma \ref{l:nu*} we get 
\begin{align*} 
  \nu(A) &= \nu_{\ast}(A \setminus 0_T) = \sup_{J \in \hT} \nu(A \setminus \pi^{-1}_J(0_J)) = \sup_{J \in \hT, \, J \supset I} \nu(\pi^{-1}_{J}(\pi^{-1}_{IJ}(B) \setminus 0_J) ) \\
  & = \sup_{J \in \hT, \, J \supset I} \rho(\pi^{-1}_{J}(\pi^{-1}_{IJ}(B) \setminus 0_J) ) = \sup_{J \in \hT} \rho(A \setminus \pi^{-1}_J(0_J)) = \rho_{\ast}(A \setminus 0_T)=\rho^0(A).
\end{align*}
Thus $\nu = \rho^0 \le \rho$ on the algebra of cylinders. By the monotone class argument the same relation holds on $\cB^T$. The proof is complete.
\qed

\bigskip

\np
{\bf Proof of Proposition \ref{p:can-spe-rep}.} The integral in \eqref{can} is well-defined by (L\ref{LM}). Given $J \in \hT$ and $a \in \R^J$, let $f(x) := \sum_{t \in J} a_t x(t)$. Then
\begin{align*} 
  \sum_{t \in J} a_t \(\widetilde{Y}_t - b(t)\)=  \sum_{t \in J} a_t I_N(x(t)) = I_N(f) + \int_{\R^T} \sum_{t \in J} a_t  x(t) [\chi(f(x)) -  \chi(x(t))] \, \nu(dx). 
\end{align*}
Using \eqref{int-cf} we get
\begin{align*} 
 \log \E \exp\(i \sum_{t \in J} a_t \(\widetilde{Y}_t - b(t)\)\) & =  \int_{\R^T} \(e^{i f(x)} - 1 - i f(x) \chi(f(x)) \)  \, \nu(dx)\\
 & \quad  + i \int_{\R^T} \sum_{t \in J} a_t x(t) [\chi(f(x)) -  \chi(x(t))] \, \nu(dx) \\
 & =  \int_{\R^T} \(e^{i f(x)} - 1 - i  \sum_{t \in J} a_t x(t) \chi(x(t)) \)  \, \nu(dx) \\
 & = \int_{\R^T} \(e^{i \langle a, x_J \rangle } - 1 - i \langle a, \lt x_J \rt \rangle  \)  \, \nu(dx)\, ,
\end{align*}
which gives \eqref{LK-fidi-P}. Therefore, $\widetilde{Y}$ is a version of $Y$. 
\qed

\bigskip

\np
{\bf Proof of Theorem \ref{t:cLM}.} (i) $\Rightarrow$ (ii). There is a set $A \in \cB^T$ such that $\nu(A)< \infty$ and $0_T \in A$. Then by \eqref{nu*}
$$
\nu(A) = \nu_{\ast}(A \setminus 0_T) = \sup_{J \in \hT} \nu(A \setminus \pi^{-1}_J(0_J) ) = \sup_{J \in \hT} \(\nu(A) - \nu(A \cap \pi^{-1}_J(0_J) )\).
$$
Hence $\inf_{J \in \hT} \nu(A \cap \pi^{-1}_J(0_J) )  =0 $ and $0_T \in A \cap \pi^{-1}_J(0_J)$, which gives (ii). 
\medskip

\np
(ii) $\Rightarrow$ (iii). Let $A_n \in \cB^T$ be such that $0_T \in A_n$ and $\nu(A_n)< n^{-1}$. There exist $T_n \in \hT_c$  and $B_n \in \cB^{T_n}$, with $0_{T_n} \in B_n$, such that $A_n = \{x: x_{T_n} \in B_n \}$. Let $T_0 = \bigcup_{n \ge 1} T_n$ and $A_0=  \{x: x_{T_0} = 0_{T_0} \}$. Since $A_0 \subset A_n$ for every  $n \ge 1$, we get $\nu(A_0)=0$,  which proves (iii). 
\medskip

\np
(iii) $\Rightarrow$ (i). Let $A_0 := \{x: x_{T_0} = 0 \}$. By the assumption $\nu(A_0)=0$. Enumerate $T_0$ as $T_0= \{t_k: k \in \N\}$. For every $k, n \in \N$ set 
$$A_{k,n} := \{x: |x(t_k)| \ge  n^{-1} \}.$$
$$
\nu(A_{k,n}) = \nu\{x: |x(t_k)| \wedge 1 \ge  n^{-1} \} \le n^2 \int_{\R^T} |x(t_k)|^2 \wedge 1 \, \nu(dx) < \infty.
$$
Since $A_0 \cup \bigcup_{k, n \in \N} A_{k,n} = \R^T$, (i) holds.
\qed

\bigskip

\np 
{\bf Proof of Theorem \ref{n-sf}.} 
Assume that $\nu$ is not $\sigma$-finite. Let $\widetilde{Y}$ be the canonical spectral representation of $Y$. That is, 
\begin{equation} \label{}
  \widetilde{Y}_t= \int_{\{x(t) \ne 0 \}}  x(t) \left[ N(dx) -  \chi(x(t)) \, \nu(dx) \right] + b(t), \quad t \in T
\end{equation}
where $N$ is a Poisson random measure on $\R^T$ with intensity $\nu$. Let $T_0 \in \hT_c$. By Corollary \ref{c:unc} there is $t_1 \notin T_0$ such that $\nu(A)>0$, where 
$$
A=\{x \in \R^T: x_{T_0} = 0, \, x(t_1) \ne 0 \}. 
$$
Define 
\begin{equation} \label{}
 \eta= \int_{A} x(t_{1}) \left[ N(dx) -  \chi(x(t_{1})) \, \nu(dx) \right]  \quad \text{and} \quad  \xi= \int_{B}  x(t_{1}) \left[ N(dx) -  \chi(x(t_{1})) \, \nu(dx) \right] + b(t_1),  
\end{equation}
where $B= \{x \in \R^T: x_{T_0} \ne 0, \, x(t_1) \ne 0 \}$.
Clearly conditions (a)--(c) hold.

Conversely, suppose (a)--(c) hold for some version $\widetilde{Y}$ of $Y$. Let $T_1= T_0 \cup \{t_1\}$ and $U= (U_t: t \in T_1)$ be given by $U_t = \widetilde{Y}_t$ when $t \in T_0$ and $U_{t_1} = \xi$; define also  $V= (V_t: t \in T_1)$ by $V_{T_0} = 0$ and $V_{t_1}=\eta$. Then $\widetilde{Y}_{T_1}=U+V$ and  processes $U$ and $ V$ are independent. Let $\nu_{\widetilde{Y}_{T_1}}$, $\nu_U$ and $\nu_{V}$ be L\'evy measures on $(\R^{T_1}, \cB^{T_1})$ of $\widetilde{Y}_{T_1}$, $U$ and $V$, respectively. We have $\nu_{\widetilde{Y}_{T_1}} = \nu_U + \nu_V$. Moreover, $\nu_V$ is concentrated on the $t_1$-axis, i.e. $\nu_V = \delta_{0_{T_0}} \otimes \nu_{\eta}$ with $\nu_{\eta}$ being a L\'evy measure of $\eta$. Hence
\begin{align*} 
  \nu\{x \in \R^T: & \ x_{T_0} = 0, \, x(t_1) \ne 0 \} = \nu\left( \pi^{-1}_{T_1}\left(0_{T_0} \times (\R \setminus \{0\})   \right)\right) \\
  & =\nu_{\widetilde{Y}_{T_1}}\left(0_{T_0} \times (\R \setminus \{0\})   \right) \ge \nu_{V}\left(0_{T_0} \times (\R \setminus \{0\})   \right) = \nu_{\eta}(\R) >0 
\end{align*}
since $\eta$ is non-degenerate. By Corollary \ref{c:unc} $\nu$ is not $\sigma$-finite, which completes the proof. \qed

\bigskip

%

\np
{\bf Proof of Lemma \ref{l:exact}.} Let $V$ be as in Definition \ref{def:rep}.  Let $f: S \mapsto (0, \infty)$ be a measurable function such that  $\int_{S} f(s) \, n(ds) < \infty$ and let  $n_1(ds) := f(s) n(ds)$ be a finite measure on $S$. Put
$$
\alpha =  \inf_{J \in \hT} \, n_1\{s\in S: V_{J}(s) = 0 \}.
$$
Then $\alpha \in [0,\infty)$ and there is $T_0 \in \hT_c$ such that $\alpha = n_1\{s\in S: V_{T_0}(s) = 0 \}$.  

Let $B \in \cB^{I}_0$, $I \in \hT$,  and let $T_1= T_0 \cup I$. Since $\{V_I \in B \} \subset \{V_{T_1}\ne 0\}$ and by the extremity of $T_0$,  \   $n_1\{V_{T_0} = 0, \,  V_{T_1} \ne 0 \} = 0$ so that $n\{V_{T_0} = 0, \,  V_{T_1} \ne 0 \} = 0$, we get
\begin{align*} \label{}
 \nu_I(B) &= n\{V_I \in B \} = n\{V_I \in B, \, V_{T_1} \ne 0 \}  \\
 & = n\{V_I \in B, \, V_{T_0} = 0, V_{T_1} \ne 0 \} +  n\{V_I \in B, \, V_{T_0} \ne 0, V_{T_1} \ne 0 \}  \\
 &= n\{V_I \in B, \, V_{T_0} \ne 0 \}\, .  
\end{align*}
Therefore, $V$ restricted to $S_0 = \{s\in S: V_{T_0}(s) \ne 0 \}$ is a representation of $\nu$ and is exact because it satisfies \eqref{cLM} (which implies (L\ref{LM0})). 
\qed

\bigskip

\np
{\bf Proof of Theorem \ref{t:LI}.}  
Consider on some probability space $(\Omega', \cF', \P')$ mutually independent centered Gaussian process $G'$ over $T$ with covariance $\Sigma$ and a Poisson random measure $N'$ on $(S, \cS)$ with intensity $n$. By Proposition \ref{p:LI-0} (with $G=G'$ and $N=N'$), we have $X\eid X' = G' + Y'$, where $Y'$ is the Poissonian part in \eqref{LI-0}. Now we restrict the index set to $T_0$ and write \eqref{LI-0} as
\begin{equation} \label{}
  X_{T_0} \eid G_{T_0}' + Y_{T_0}' = f(G'_{T_0}, N'_{\cS_0})
\end{equation}
where $N'_{\cS_0}$ is the restriction of $N'$ to  a countable algebra $\cS_0$ that generates $\cS$ modulo $n$. Indeed, since the stochastic integral with respect to $N'$ is a limit of integrals of simple functions and $N'$ on $\cS$ can be approximated by $N'$  on $\cS_0$, the right hand side of \eqref{LI-0} (restricted to $T_0$) can be represented as a Borel function $f$ of a random element $(G'_{T_0}, N'_{\cS_0})$ taking values in a Polish space $\R^{T_0} \times [0,\infty]^{\cS_0}$. By \cite[Corollary 6.11]{Kallenberg(01)} there exists a random element $\((G_{t})_{t\in T_0}, (N(A))_{A\in \cS_0}\)$ on the original probability space of $X$ such that 
\begin{equation} \label{Kal1}
  X_{T_0} = f((G_{t})_{t\in T_0}, (N(A))_{A\in \cS_0}) \quad a.s.
\end{equation}
and $\((G_{t})_{t\in T_0}, (N(A))_{A\in \cS_0}\) \eid (G'_{T_0}, N'_{\cS_0})$. Since 
$(N(A))_{A\in \cS_0}$ is a Poisson random measure on the algebra $\cS_0$, independent of $(G_{t})_{t\in T_0}$, it extends uniquely to a Poisson random measure $N$ on $(S, \cS, n)$, which is also independent of $(G_{t})_{t\in T_0}$. Therefore, \eqref{Kal1} establishes \eqref{LI} for $t \in T_0$ (see \cite[Theorem 5.2]{R-R89} for more details).

Since $X$ is separable in probability,  for every $t \in T$ there exists $\tau_n \in T_0$ such that $X_{\tau_n} \cip X_t$. By a symmetrization inequality,  $G_t := \lim_{n \to \infty} G_{\tau_n}$ exists in probability and $G=\{G_t \}_{t\in T}$ is independent of $N$. Having $G$ and $N$ constructed, we use Proposition \ref{p:LI-0} again to state that
\begin{equation} \label{}
  X_t'' := G_t + \int_{S} V_t \, \big(dN - \chi(V_t) \, dn\big) + b(t), \quad t\in T
\end{equation}
is a version of $X$. Since  \ $X_t = X_t''$ a.s. for each $t \in T_0$, and $T_0$ is a common separant for both  $X$ and $X''$, we get $X_t=X_t''$  a.s. for all $t \in T$. This establishes \eqref{LI} and completes the proof. 
\qed

\bigskip

\np
{\bf Proof of Theorem \ref{transfer}.}
Let $f: S \mapsto (0,1]$ be a measurable function such that $\int_{S} f(s) \, n(ds) <\infty$.  Consider a finite measure $n_0(ds) := f(s) n(ds)$  on $S$. Let $\xi^k=(\xi^{k}_{t})_{t\in T}$ be an i.i.d.  sequence of processes over $T$ with the common distribution $\theta^{-1} n_0\circ V^{-1}$, where $\theta= n_0(S)$, so that 
$$
n_0\circ V^{-1}(B) = \theta \P( \xi^k \in B), \quad B \in \cB^T.   
$$
 Let $Y_t:= \sum_{k=1}^{\eta} \xi_t^k$ be a compound Poisson  process, where $\eta$ is a Poisson random variable with mean $\theta$ and independent of $\{\xi^k \}$. Put $\nu_0:=n_0\circ V^{-1}$ and notice that for every  $B \in \cB^T$
$$
  \nu_0(B) = \int_{S} \1_B(V(s)) f(s) \, n(dx) \le \int_{S} \1_B(V(s)) \, n(dx) = \nu(B).
$$
Therefore, $\nu_0$ is equivalent to $\nu$ and $\nu_0 \le \nu$. Since $\nu$ is $\sigma$-finite, $\nu$ satisfies condition \eqref{cLM}, so does $\nu_0$. Therefore, $\nu_0$ is a L\'evy measure of $Y$ (see Example \ref{e:CP}). Suppose that the process $X$ has the generating triplet $(\Sigma, \nu, b)$.  Let $Z=(Z_{t})_{t\in T}$ be an infinitely divisible process independent of $\{\eta, \xi_t^k: t\in T, \ k\in \N \}$ and with the generating triplet $(\Sigma, \nu-\nu_0, c)$, where $c$ is a shift function such that 
$$
X \eid Y + Z.  
$$
By Lemma \ref{l:reg} given below, there exists $\widetilde{X}$ with all paths in $U$ such that $\widetilde{X}_t = Y_t + Z_t$ a.s. for each $t \in T$. We will now check that $Z$ satisfies  \eqref{reg}. For every $A \subset \R^T \setminus U$, $A \in \cB^T$ we get
\begin{align*} 
  \P(Z \in A) = e^{\theta} \P(Z \in A, \, \eta=0) = e^{\theta} \P(Y+Z \in A, \, \eta=0) \le  e^{\theta} \P(\widetilde{X} \in A) =0. 
\end{align*}
Thus, there exists $\widetilde{Z}$ with all paths in $U$ such that $\widetilde{Z}_t = Z_t$ a.s. for each $t \in T$. Hence, by our assumption on $U$, \  $\widetilde{Y}:= \widetilde{X} - \widetilde{Z}$ is a modification of $Y$ with all paths in $U$. Consider the representation $V$ as a stochastic process under the probability measure $\theta^{-1} n_0$, so we have $V \eid \xi^1$. For any set $A$ as above we have
\begin{align*} 
\theta^{-1} n_0( V \in A) & =  \P( \xi^1 \in A) = \theta^{-1} e^{\theta} \P( \xi^1 \in A, \, \eta =1) = \theta^{-1} e^{\theta} \P( \widetilde{Y} \in A, \, \eta =1) = 0.
\end{align*}
By Lemma \ref{l:reg} there exists a process $\widetilde{V}=(\widetilde{V}_{t})_{t\in T}$ with all paths in $U$ such that 
$$
\theta^{-1} n_0\{s: \widetilde{V}_{t}(s)\ne V_t(s) \}=0 \quad \text{for all } t \in T.
$$
Since the measures $n$ and $\theta^{-1} n_0$ are equivalent, this proof is complete.
\qed 

\bigskip

\begin{lemma}\label{l:reg}
	Let $X=\(X_t \)_{t\in T}$ be a stochastic process and let $U \subset \R^T$. Assume that $(U, \mathcal{U})$  is a Borel space for the $\sigma$-algebra $\mathcal{U}= \mathcal{B}^T \cap U$. Then there exists a process $\widetilde{X}$ with all paths  in $U$  such that $\widetilde{X}_t=X_t$ a.s. for every $t \in T$ if and only if 
\begin{equation} \label{reg}
  \P( X \in A)=0 \quad \text{for all } A \subset \R^T \setminus U, \ A \in \cB^T.
\end{equation}
\end{lemma}
\np
{\bf Proof of Lemma \ref{l:reg}.} The necessity of \eqref{reg} is obvious, so we will prove its sufficiency.  Define for every $B \in \mathcal{U}$ 
$$
\mu(B) := \P(X \in A), \quad \text{when } A \cap U = B, \   A \in \cB^T.
$$
It is routine to check that under \eqref{reg} $\mu$ is a well-defined probability measure on $(U, \mathcal{U})$. Let $Y_t(u) :=u(t)$, $t\in T$, $u \in U$. Then $Y=\(Y_t \)_{t\in T}$ is a stochastic process defined on $(U, \mathcal{U}, \mu)$ with paths in $U$. For every $A \in \cB^T$
$$
\mu(Y \in A) =  \mu(Y \in A \cap U) = \P(X \in A), 
$$
so that $Y \eid X$. By \cite[Lemma 3.24]{Kallenberg(01)}, $X$ has a modification $\widetilde{X}$ whose paths lie in $U$ such that $\widetilde{X}_t=X_t$ a.s. for each $t\in T$. \qed 

\bigskip

As we have mentioned in Section \ref{s:iso}, Theorem \ref{t:iso0} is a consequence of Theorem \ref{t:iso1}, which itself is deduced from Theorem \ref{t:iso2}. Therefore, we begin with Theorem \ref{t:iso2}.

\medskip

\np
{\bf Proof of Theorem \ref{t:iso2}.} Notice that $\E N(q) = \int_{S} q \, dn =1$. Let $Y= X -G$ be the Poissonian part of $X$. It is enough to prove \eqref{iso2} for $F$ of the form $F(x)= \exp\{i \sum_{j=1}^{n} a_j x(t_j) \}$, where $a_1,\dots,a_n \in \R $, $t_1,\dots,t_n \in T$, $n \ge 1$. We have $F(X)=F(G)F(Y)$, where
$$
F(Y) =\exp\left\{i \sum_{j=1}^{n} a_j \left[ \int_{S} V_{t_j} \left(dN - \chi(V_{t_j})\, dn\right)+ b(t_j)\right]\right\} = H(N)
$$
Here $N=(N(A))_{A \in \cS}$ is viewed as a stochastic process and $H: \R^{\cS} \mapsto \R$ is a measurable functional. 
Using the independence and Mecke-Palm formula we get
\begin{align*} 
  \E [F(X)N(q)] &= \E [F(G)H(N) N(q)] = \E [F(G)]\, \E [H(N) N(q)] \\
  & = \E [F(G)] \,  \E \int_{S}\, q(s) H(N) \, N(ds) \\
  &= \E [F(G)]\, \int_{S} \E \left[ q(s) H(N + \delta_s) \right] \, n(ds)\\
  &= \E [F(G)]\, \int_{S} \E \left[ q(s) F(V(s)) H(N) \right] \, n(ds)\\
  &= \E [F(G)]\, \E [H(N)] \int_{S} F(V(s))\, q(s)\, n(ds)\\
   &= \E [F(X)]\,  \int_{S} F(V(s))\, q(s)\, n(ds)\\
   &= \int_{S} \E  F(X+ V(s)) \,   \, q(s)\, n(ds)\, .
\end{align*}
This establishes \eqref{iso2}.

To prove \eqref{iso3} notice that $X$ and $N(q)$ are jointly infinitely divisible. Let $\theta$ be an isolated point of $T$ and put $T_{\theta} = T \cup \{\theta\}$. Consider an infinitely divisible process $\bar{X} = (\bar{X}_{t})_{t\in T_{\theta}}$ given by
$$
\bar{X}_t = \begin{cases}
    X_t    &  t \in T   \\ 
     N(q)      & t=\theta. 
\end{cases} 
$$
The L\'evy measure of $\bar{X}$ has a representation $\bar{V}$ on $(S, \cS, n)$ of the form
$$
\bar{V}_t = \begin{cases}
    V_t    &  t \in T   \\ 
     q      & t=\theta. 
\end{cases} 
$$
Let $H: \R^{T_{\theta}} \mapsto \R$ be given by $H(x) = F(x_{T}) \frac{1}{x_{\theta}} \1( x_{\theta} >0)$, $x \in T_{\theta}$, where $F$ is as above. Applying \eqref{iso2} we get
\begin{align} 
  \E [F\left( \(X_t\)_{t\in T}\right); \, N(q)>0] &=  \E [H\left( \(\bar{X}_t\)_{t\in T_{\theta}}\right); \, N(q)] \\
  &= \int_{S} \E  H\left(\(\bar{X}_t+ \bar{V}_t(s)\)_{t\in T_{\theta}} \right) \, q(s)\, n(ds) \\
  &= \int_{S} \E  F\left(\(X_t+ V_t(s)\)_{t\in T} \right) \frac{\1(N(q)+ q(s) >0)}{N(q) + q(s)} \, q(s)\, n(ds) \\
 &= \int_{S} \E \big[ F\left(\(X_t+ V_t(s)\)_{t\in T} \right); \, (N(q)+ q(s))^{-1}\big] \,q(s)\, n(ds)\, , 
\end{align}
which shows \eqref{iso3}.

The last formula \eqref{iso4} follows from the previous \eqref{iso3} since
\begin{align}\label{N>0} 
  \P(N(q)>0) &= 1- \lim_{u \to \infty} \E e^{-u N(q)} = 
1-  \lim_{u \to \infty} \exp\left(\int_{S} [e^{-u q(s)} -1 ] \, n(ds)\right) \\
  &= 1- \exp\left(- n\{ q(s) >0\} \right).
\end{align}
The proof is complete. 
\qed

\bigskip

\np
{\bf Proof of Theorem \ref{t:iso1}.}
First we will show that part {\bf (a)} follows from part {\bf (b)} of this theorem. Indeed, suppose $\cL(Z) \ll \nu$, so that $\cL(Z) \ll \nu + \delta_{0_T}$. Let $q = \frac{d\cL(Z)}{d(\nu+ \delta_{0_T})}$. Since $\nu$ is $\sigma$-finite, $\nu\{x: x_{T_0}=0 \}=0$ for some $T_0 \in \hT_c$, so that $\P(Z_{T_0} = 0 )=0$.  
We have
\begin{equation} \label{0-cont}
 0= \P(Z_{T_0} = 0)= \int_{\{x: x_{T_0}=0 \}} q(x) \, \nu(dx) + q(0_T) \delta_{0_{T}}(\{x: x_{T_0}=0 \}) = q(0_T). 
\end{equation}
Hence $q(0_T) =0$, in which case \eqref{iso1b} becomes \eqref{iso1} and \eqref{iso1'b} becomes \eqref{iso1'}. Therefore, we only need to prove {\bf (b)}.

Let $S= \R^T$ and $\cS= \cB^T$. Consider $N_1 = N + \eta \delta_{0_T}$, where $\eta$ is a Poisson random variable with mean 1 independent of $N$ and $G$. $N_1$ is a Poisson random measure on $(S, \cS)$ with intensity $n_1 = \nu + \delta_{0_T}$. Notice that \eqref{can1} still holds after replacing $N$ by $N_1$ and $\nu$ by $n_1$. Therefore, $V_t(x)=x(t)$ is a representation of $\nu$ on $(S, \cS, n_1)$. Using Theorem \ref{t:iso2} we get
\begin{align*} 
  \E F\left(\(X_t+Z_t\)_{t \in T} \right) &= \int_{S} \E F\left(\(X_t+x(t)\)_{t \in T} \right) \, q(x) \, n(dx) \\
  &= \E \left[F\left( \(X_t\)_{t\in T}\right); \, N_1(q) \right] 
  = \E \left[F\left( \(X_t\)_{t\in T}\right); \, N(q) + \eta q(0_T) \right] \\
  &= \E \left[F\left( \(X_t\)_{t\in T}\right); \, N(q) \right] + \E \left[F\left( \(X_t\)_{t\in T}\right) \right]\, q(0_T)\, , 
\end{align*}
by the independence of $N$, $G$ and $\eta$, and $\E \eta=1$. This proves \eqref{iso1b}.


Since $\cL(Z) \ll n_r:=\nu + r \delta_{0_T}$ for every  $r>0$, we may consider $q_r := \frac{d\cL(Z)}{d n_r}$, so that $q_1=q$. Because $\nu$ and $\delta_{0_T}$ are singular, $q_r(0_T)= r^{-1}q(0_{T})$ and $q_r=q$  $\nu$-a.e. Therefore,  we can take $q_r = q\1_{U^c} + r^{-1}q(0_{T})\1_{U}$ as a version of $q_r$, where $U$ is any set such that  $0_T \in U \in \cB^T$ and $\nu(U)=0$.

Let $\eta_r$ be a Poisson random variable with mean $r$, independent of $N$ and $G$, and let $N_r = N + \eta_r \delta_{0_T}$. By the same argument as above, with $n_1$ replaced by $n_r$ and $N_1$ replaced by $N_r$, we use \eqref{iso3} of Theorem \ref{t:iso2} to get
\begin{align*} 
  \E & \left[F\(\(X_t\)_{t \in T}\); \, N_r(q_r) >0\right] \\
  &\qquad =  \E \left[F\left(\(X_t+Z_t\)_{t \in T} \right)\(N_r(q_r) + q_r(Z) \)^{-1} \right]
\end{align*}
which can written as
\begin{align} \label{}
  \E & \left[F\(\(X_t\)_{t \in T}\); \, N(q) + r^{-1} \eta_r q(0_T) >0\right] \\
  &\qquad =  \E \left[F\left(\(X_t+Z_t\)_{t \in T} \right)\(N(q) + r^{-1} \eta_r q(0_T) + q(Z)\1_{U^c}(Z) + r^{-1} q(0_T) \1_{U}(Z)\)^{-1} \,  \right].
\end{align}
Letting $r \to \infty$, and using that $r^{-1} \eta_r \cip 1$, we get
\begin{align} \label{}
  \E & \left[F\(\(X_t\)_{t \in T}\)\right] = \E \left[F\(\(X_t\)_{t \in T}\); \, N(q) + q(0_T) >0\right]\\
  &\qquad =  \E \left[F\left(\(X_t+Z_t\)_{t \in T} \right)\(N(q) + q(0_T) + q(Z)\1_{U^c}(Z) \)^{-1} \,  \right]\\
  &\qquad =  \E \left[F\left(\(X_t+Z_t\)_{t \in T} \right)\(N(q) + q(Z) + q(0_T)\1_{U^c}(Z) \)^{-1} \,  \right]\, 
\end{align}
because $q(Z)\1_U(Z) = q(0_T)\1_U(Z)$ a.s. 

Finally, $\cL(X+Z) \ll \cL(X)$ follows from \eqref{iso1}. Conversely, notice that $N(q) + q(0_T)>0$ a.s. if either $q(0_T)>0$ or, by \eqref{N>0}, $\nu\{x: q(x) >0 \}= \infty$. In these cases \eqref{iso1'} gives $\cL(X) \ll \cL(X+Z)$. 
The proof is complete.
\qed
   
\bigskip  

\np
{\bf Proof of Theorem \ref{t:iso0}:}
Without loss of generality, we may assume that $X= G+Y$, where $G$ is s centered Gaussian process, $Y$ is the Poissonian part of $X$ given by a canonical spectral representation \eqref{can1} relative to a Poisson random measure $N$, where $G$, $N$, and $Z$ are independent. By Theorem \ref{t:iso1}, 
$$
\E F\left(\(X_t+Z_t\)_{t \in T} \right) = \E \left[ F\left(\(X_t\)_{t \in T} \right); \, N(q) \right]\, ,
$$
which implies
$$
\cL(X+Z)  \ll \cL(X).
$$
Therefore $\E F\left(\(X_t+Z_t\)_{t \in T} \right) = \E \left[ F\left(\(X_t\)_{t \in T} \right); \, g(X) \right]$, where $g(x) = \frac{d\cL(X+Z)}{d\cL(X)}(x)$, $x\in \R^T$.
\qed

\bigskip

\np 
{\bf Proof of Lemma \ref{l:s-iso}.}
We follow, with some necessary modifications,  arguments from \cite[Lemma 3.1]{Eisenbaum08}. Assume $(i)$. Then the Laplace transform of $Y$ is of the form
$$
 \phi(\alpha_1,\dots,\alpha_n)=\E \exp\{-\sum_{i=1}^{n} \alpha_i Y_i\}
  = \exp\big\{-\sum_{i=1}^{n} \alpha_i c_i -\int_{\R_+^{n}}
 (1-e^{-\sum_{i=1}^{n}\alpha_i y_i}) \nu(dy) \big\}.
$$
where $\alpha_i, c_i\ge 0$. We have
\begin{align} \label{e0}
  \frac{\partial}{\partial \alpha_k} \phi(\alpha_1,\dots,\alpha_n) = 
 -\phi(\alpha_1,\dots,\alpha_n)\big[c_k + \int_{\R_+^{n}}
 e^{-\sum_{i=1}^{n}\alpha_i y_i} \, y_k \nu(dy) \big].
\end{align}
Hence $\theta_k=\E(Y_k)= c_k + \int_{\R_+^{n}} y_k \nu(dy)$. Let $Z^k$ be a vector in $\R_{+}^n$ independent of $Y$ whose  distribution is given by
$$
 \mathcal{L}(Z^k)(dy) = \frac{c_k}{\theta_k} \delta_{(0,\dots, 0)}(dy) +  \frac{y_k}{\theta_k} \, \nu(dy).
$$
Using \eqref{e0} we obtain
$$
\E \big[\exp\{-\sum_{i=1}^{n} \alpha_i Y_i\};\, \theta_k^{-1}Y_k\big] = \E \big[\exp\{-\sum_{i=1}^{n} \alpha_i (Y_i + Z^k_i)\}\big],
$$
which yields (ii).
\medskip

Assume $(ii)$. 
We will prove that $Y$ is infinitely divisible with the L\'evy measure and drift given by  \eqref{R1}-\eqref{R2}. To this end, we first show that 
for any bounded measurable functional $F: \R^n \mapsto \R$  and $j, k \le n$
\begin{equation}\label{e1} 
\theta_j  \E [F(Z^j) Z^j_k] =  \theta_k \E[ F(Z^k)Z^k_j].
\end{equation}
It is enough to show \eqref{e1} for $F(y)=\exp\{- \sum_{i=1}^{n} \alpha_i y_i\}$, where $\alpha_i > 0$. Using  \eqref{E2} twice and independence we get
\begin{align*} 
  \E [F(Y) Y_j; Y_k] &= \theta_k \E[ F(Y+Z^k) (Y_j+Z^k_j)] \\
  &= \theta_k \E[ F(Y+Z^k); Y_j)] + \theta_k \E[ F(Y+Z^k)Z^k_j] \\
  &= \theta_k \E[ F(Y); Y_j)]\E[ F(Z^k)] + \theta_k \E[ F(Z^k)Z^k_j] \E[ F(Y)] \\
  &= \theta_j \theta_k \E[F(Y)] \E[F(Z^j)]\E[ F(Z^k)] + \theta_k \E[ F(Z^k)Z^k_j] \E[ F(Y)].
\end{align*}
Since interchanging $j$ and $k$ does not change the first term in these equations, equating the final terms after the interchange gives \eqref{e1}.
Taking $F(y)=\1_{\{y_j=0 \}}$ in \eqref{e1} yields
$$
\theta_j  \E [\1_{\{Z^j_j=0 \}} Z^j_k] =  \theta_k \E[\1_{\{Z^k_j=0 \}} Z^k_j]= 0.
$$
This implies that for every $j, k \le n$,
\begin{equation} \label{e2}
  \{Z^j_k >0 \} \subset \{Z^j_j >0 \} \quad \text{a.s.} 
\end{equation}
Applying \eqref{e1} to $F(y)(y_j y_k)^{-1}\1_{\{y_j y_k >0 \}}$ in the place of $F$, where $F\ge 0$, and taking into account \eqref{e2}, we obtain  
\begin{equation} \label{e3}
 \theta_j  \E [F(Z^j) \1_{\{Z^j_k>0 \}} (Z^j_j)^{-1}] =  \theta_k \E[ F(Z^k) \1_{\{Z^k_j>0 \}} (Z^k_k)^{-1}]. 
\end{equation}

Now we are ready to prove (i) together with \eqref{R1}-\eqref{R2}.  For $n=1$, we get from \eqref{E2} 
$$
\E[e^{-\alpha Y} Y]= \theta \E e^{-\alpha (Y+Z)} = \theta \E e^{-\alpha Y} \E e^{-\alpha Z},
$$
which yields 
$$
\frac{d}{d \alpha} \log \E[e^{-\alpha Y}] = - \theta \E e^{-\alpha Z}.
$$
Therefore,
\begin{align*} 
\E[e^{-\alpha Y}] &= \exp\{-\theta \E \int_{0}^{\alpha} (\1_{\{Z=0 \}} +  e^{-sZ}\1_{\{Z>0 \}})\, ds \}\\
& =\exp\{- \alpha \theta \P(Z=0) - \theta\int_{0}^{\infty} (1-e^{-\alpha y}) \1{\{y>0 \}} y^{-1} \, \mathcal{L}(Z)(dy) \},  
\end{align*}
which shows our claim. 

We proceed by induction. Assuming (ii), suppose that (i) and \eqref{R1}-\eqref{R2} hold for $n-1$. Since $\tilde{Y}:=Y_{\{1,\dots,n-1 \}}$ and $\tilde{Z}^j:=Z^j_{\{1,\dots,n-1 \}}$ ($j\le n-1$) satisfy (ii) for $n-1$ in the place of $n$, by the induction hypothesis $\tilde{Y}$ is infinitely divisible with the Laplace transform  
$$
\tilde{\phi}(\alpha_1,\dots, \alpha_{n-1})=  \exp\{-\sum_{i=1}^{n-1} \alpha_i \tilde{c}_i -\int_{\R_+^{n-1}} (1-e^{-\sum_{i=1}^{n-1}\alpha_i y_i}) \tilde{\nu}(dy) \},
$$
where 
$
\tilde{c}= (\theta_1 \P(Z^1_1=0),\dots, \theta_n \P(Z^{n-1}_{n-1}=0))
$
and
$$
\tilde{\nu}(dy) = \sum_{k=1}^{n-1} \theta_k \1_{\{y_1= \cdots =y_{k-1}=0, y_{k}>0 \}} y_{k}^{-1} \mathcal{L}(\tilde{Z}^k)(dy). 
$$
Let $\phi(\alpha_1,\dots,\alpha_n)=\E \exp\{-\sum_{i=1}^{n} \alpha_i Y_i\}$. Proceeding as before, we get
\begin{align*} 
  \frac{\partial}{\partial \alpha_n} \log \phi(\alpha_1,\dots,\alpha_n) = - \theta_n \E \exp\{-\sum_{i=1}^{n} \alpha_i Z_i^{n}\}.
\end{align*}
Hence
\begin{align}\label{e4} 
  \phi(\alpha_1,\dots,\alpha_n) &= \tilde{\phi}(\alpha_1,\dots, \alpha_{n-1}) \exp\big\{-\theta_n \E \big[e^{-\sum_{i=1}^{n-1} \alpha_i Z_i^{n}}\int_{0}^{\alpha_n} e^{-s Z_n^{n}}\, ds\big] \big\}.
\end{align}
Notice that $Z_i^{n}=0$ a.s. on the set $\{Z^n_n=0 \}$ by \eqref{e2}.
Therefore,  the exponent of the last term on the right hand side of \eqref{e4} equals
\begin{align*} 
  -\theta_n & \E \big[\alpha_n \1_{\{Z^n_n=0 \}} + e^{-\sum_{i=1}^{n-1} \alpha_i Z_i^{n}} (1- e^{-\alpha_n Z_n^{n}}) \1_{\{Z_n^n>0\}}    
  (Z_n^n)^{-1}\big] \\
  &=-\alpha_n c_n -\theta_n \sum_{k=1}^{n} \E \big[e^{-\sum_{i=1}^{n-1} \alpha_i Z_i^{n}} (1- e^{-\alpha_n Z_n^{n}}) \1_{\{Z^n_1= \cdots =Z^n_{k-1}=0, \, Z_k^n>0\}} (Z_n^n)^{-1}  \big]
\end{align*}
which, after applying \eqref{e3} and noticing that the term $Z_n^k>0$ can be replaced by $Z_k^k>0$, gives us
\begin{align*} 
  & =-\alpha_n c_n - \sum_{k=1}^{n} \theta_k\E \big[e^{-\sum_{i=1}^{n-1} \alpha_i Z_i^{k}} (1- e^{-\alpha_n Z_n^{k}}) \1_{\{Z^k_1=0= \cdots =Z^k_{k-1}=0, \, Z_k^k>0\}} (Z_k^k)^{-1}  \big] \\
  &=-\alpha_n c_n - \sum_{k=1}^{n} \theta_k \int_{\R_+^n} \big[e^{-\sum_{i=1}^{n-1}\alpha_i y_i} - 
  e^{-\sum_{i=1}^{n}\alpha_i y_i}\big] \, \1_{\{y_1= \cdots =y_{k-1}=0, y_{k}>0 \}} y_{k}^{-1} \mathcal{L}(Z^k)(dy) \\
  &=-\alpha_n c_n + \int_{\R_+^n} \big(1-e^{-\sum_{i=1}^{n-1}\alpha_i y_i}\big) \, \nu(dy) - 
  \int_{\R_+^n} \big(1-e^{-\sum_{i=1}^{n}\alpha_i y_i}\big) \, \nu(dy).
\end{align*}
Substituting the above into \eqref{e4} completes the proof.
\qed

\bigskip

\np 
{\bf Proof of Proposition \ref{p:s-iso}.}
The form of the drift follows from Lemma \ref{l:s-iso}.  
Let  $I=\{t_1,\dots,t_n \} \subset T$.  We have
\begin{align} 
  \int_{\R_{+}^T}  \big(1- & e^{-\sum_{i=1}^{n} \alpha_i y(t_i)}\big) \, \nu(dy)= \sum_{k\ge 1} \int_{\R_{+}^T} \big(1- e^{-\sum_{i=1}^{n} \alpha_i y(t_i)}\big) \, \nu_k(dy) \\
  &=\sum_{k\ge 1}\sum_{j=1}^n \int_{\R_{+}^T}  \big(1- e^{-\sum_{i=1}^{n} \alpha_i y(t_i)}\big)\1_{\{y(t_1)=\cdots =y(t_{j-1})=0,\, y(t_j)>0 \}} \, \nu_k(dy), \label{e5}
\end{align}
and by \eqref{e3}, for every  $j, k$, 
\begin{align*} 
 \int_{\R_{+}^T} & \big(1- e^{-\sum_{i=1}^{n} \alpha_i y(t_i)}\big)\1_{\{y(t_1)=\cdots =y(t_{j-1})=0,\, y(t_j)>0 \}} \, \nu_k(dy)\\
 & = \theta(s_k)
 \E\Big[ \big(1- e^{-\sum_{i=1}^{n} \alpha_i Z^{s_k}_{t_i}}\big)\1_{\{Z^{s_k}_{s_1}=\cdots= Z^{s_k}_{s_{k-1}}=0,\, Z^{s_k}_{s_k}>0,\, Z^{s_k}_{t_1}=\cdots=Z^{s_k}_{t_{j-1}}=0,\, Z^{s_k}_{t_j}>0 \}} (Z^{s_k}_{s_k})^{-1}\Big] \\
 &=  \theta(t_j)
 \E \Big[\big(1- e^{-\sum_{i=1}^{n} \alpha_i Z^{t_j}_{t_i}}\big)\1_{\{Z^{t_j}_{s_1}=\cdots = Z^{t_j}_{s_{k-1}}=0,\, Z^{t_j}_{s_k}>0,\, Z^{t_j}_{t_1}=\cdots =Z^{t_j}_{t_{j-1}}=0,\, Z^{t_j}_{t_j}>0 \}} (Z^{t_j}_{t_j})^{-1}\Big]
\end{align*}
Substituting this into \eqref{e5} and summing over $k$ gives 
\begin{align*} 
\int_{\R_{+}^T} & \big(1- e^{-\sum_{i=1}^{n} \alpha_i y(t_i)}\big) \, \nu(dy)\\
&=
  \sum_{j=1}^n \theta(t_j)
 \E \big(1-  e^{-\sum_{i=1}^{n} \alpha_i Z^{t_j}_{t_i}}\big)\1_{\{Z^{t_j}_{T_0} \ne  0 \}} \1_{\{Z^{t_j}_{t_1}=0,\dots,Z^{t_j}_{t_{j-1}}=0, Z^{t_j}_{t_j}>0 \}} (Z^{t_j}_{t_j})^{-1}
\end{align*}
In view of Lemma \ref{l:s-iso},  we now need to  show that  $\{Z^{t_j}_{T_0} \ne 0 \} = \{Z^{t_j}_{t_j}>0 \}$ a.s. To this aim, we first notice that $\{Z^{t_j}_{T_0} \ne 0 \} \subset \{Z^{t_j}_{t_j}>0 \}$ a.s. by \eqref{e2}. Then, choose $(s_{k_n})_{n\ge 1} \subset T_0$ such that $Y_{s_{k_n}} \cip Y_{t_j}$. Using \eqref{s-iso} for $F(y)= \exp(-\alpha (y(s_{k_n})- y(t_j)))$, $\alpha >0$ we get $Z^{t_j}_{s_{k_n}} \cip Z^{t_j}_{t_j}$ as $n \to \infty$. Hence
$$
\P(Z^{t_j}_{t_j}>0) \le \liminf_{n \to \infty} \P(Z^{t_j}_{s_{k_n}}>0) \le \P(Z^{t_j}_{T_0}>0) \le \P(Z^{t_j}_{t_j}>0).
$$
Since $\nu$ determines all finite dimensional distributions of $Y$ and clearly $\nu\{y: y_{T_0} =0 \} =0$, $\nu$ is the L\'evy measure of $Y$. 

To complete the proof, consider a nonnegative infinitely divisible process $(Y_t)_{t\in T}$ with finite mean $\theta(t)= \E Y_t$. Let $c$ and $\nu$ be the drift and L\'evy measure of $Y$, respectively. 
Similarly to the proof of Lemma \ref{l:s-iso}(i), we check that \eqref{iso-tr} satisfies \eqref{s-iso}. The proof is complete.
  \qed

\bigskip

\np 
{\bf Proof of Proposition \ref{p:Lp0}.}
Let $X$ be determined by \eqref{Lp-LK}, where we take $\lt v\rt = v \1_{[-1,1]}(v)$ for concreteness. 
 Given $h>0$, let $F: \R^{[0,h]} \mapsto \R$ be defined by $F(x)=f(x(h))$. 
 By \eqref{Lp-iso4} we have
 \begin{equation} \label{}
  h^{-1}\E [f(X_h); W_h >0] = \int_{\R } \E \big[ f\(X_h+ v\); \, (W_h+ q_1(v))^{-1}\big] \,q_1(v)\, \rho(dv)\,. 
\end{equation}
Choose $q_1 = m^{-1} \1_{\{|v|>\delta\}}$, where $\delta \in (0,1)$ is fixed and  $m :=\rho\{|v|>\delta \}>0$. Then we have
\begin{equation} \label{}
  h^{-1}\E [f(X_h)] =  h^{-1}\E [f(X_h); W_h =0] + \int_{ |v| > \delta } \E \big[ f\(X_h+ v\); \, (mW_h+ 1)^{-1}\big] \, \rho(dv)\, .
\end{equation}
Hence
\begin{align*} 
 | h^{-1} \E [f(X_h)]   - \int_{\R } f\(v\)  \,  \rho(dv) | & \le h^{-1}|\E [f(X_h); W_h =0]| + \int_{|v|\le \delta} |f\(v\)|  \,  \rho(dv) \\
 & + \int_{\{|v|> \delta\}} \E \big| f\(X_h+ v\) \(m W_h+ 1\)^{-1} - f(v)\big|  \, \rho(dv)  \\
 & :=  K_1(h, \delta) + K_2(\delta) + K_3(h, \delta)\, .
\end{align*}
Notice that if $W_h=0$, then $X_h=X^{\delta}_h$, where $X^{\delta}_h= X_h - \sum_{s\le h} \Delta X_s \1(|\Delta X_s| > \delta)$.   $X_h^{\delta}$ is a L\'evy process with variance $\sigma^2$ of its Brownian motion component $\{G_t\}$, L\'evy measure $\rho = \rho_{|\{|x|\le \delta \}}$, and the shift $c^{\delta}= c - \int_{\delta < |x|\le 1} x \, \rho(dx)$, so that $\E X^{\delta}_h = c^{\delta}h$. Put $Y_h^{\delta}=X_h^{\delta}-G_h$.  By the assumption, $|f(x)| \le x^2 k(x)$, where $0\le k(x) \le C$ is a bounded function with $\lim_{x \to 0} k(x)=0$, or only a bounded function when $\sigma^2=0$. We get
\begin{align*} 
  K_1(h,\delta) &\le  h^{-1}\E \(|X^{\delta}_h|^2k(X^{\delta}_h)\) \le 2 h^{-1}\E \(|G_h|^2k(X^{\delta}_h)\) + 2Ch^{-1}\E \(|Y^{\delta}_h|^2\) \\
  & \le 2 h^{-1} \(\E |G_h|^4\)^{1/2} \(\E |k(X^{\delta}_h)|^2\)^{1/2} + 
  2C h^{-1}\(\var\(Y^{\delta}_h\) + |\E Y^{\delta}_h|^2\) \\
 &= 2\sqrt{3} \sigma^2\(\E |k(X^{\delta}_h)|^2\)^{1/2} + 2C\var\(Y^{\delta}_1\) + 2C |c^{\delta}|^2 h \ \stackrel{h \to 0}{\longrightarrow} \ 2C \int_{\{|v| \le \delta \}} v^2 \, \rho(dv)\, .
\end{align*}
We see that when $\sigma^2=0$ we only need $k$ to be bounded. 
Then $K_2(\delta) \le C \int_{\{|v| \le \delta \}} v^2 \, \rho(dv)$ and $K_3(h,\delta) \to 0$ as $h \to 0$ because $X_h, W_h \to 0$ a.s. and continuous on a set of full $\rho$-measure.  
Combining the above estimates we get
$$
\limsup_{h \to 0}  | h^{-1} \E [f(X_h)]   - \int_{\R } f\(v\)  \,  \rho(dv) | \le 3C \int_{\{|v| \le \delta \}} v^2 \, \rho(dv)\, .
$$
Letting $\delta \to 0$ gives \eqref{LP3}.  

Finally, notice that if $X$ is a subordinator, then in the term $K_3(h,\delta)$, $X_h +v$ approaches $v$ from the right. Thus, if $f$ is right-continuous on a set of full $\rho$-measure we get also $\lim_{h\to 0}K_1(h, \delta)= 0$. This completes the proof.
\qed

\bigskip

\np
{\bf Proof of Theorem \ref{t:iso-s}.} Recall that $Y$ has the generating triplet $(0, \nu, b)$ and $V$ is a representation of $\nu$ on a $\sigma$-finite measure space $(S, \mathcal{S}, n)$. $\(\xi _j \)_{j\in \N}$ is an i.i.d. sequence of random elements in $S$ with the common distribution $g(s) n(ds)$, where $g>0$ $n$-a.e.  Finally, $\(\Gamma_j \)_{j\in \N}$ is a sequence of partial sums of i.i.d. standard exponential random variables independent of 
 $\(\xi_j \)_{j\in \N}$. 

We will show the isomorphism identities by a reduction to the stochastic integral case. To this aim, consider $\bar{S}=\{(s,r) \in S\times \R_{+}: 0\le r \le g(s)^{-1} \}$ with the measure $\bar{n}(ds,dr) := g(s) n(ds) dr$ and let $\bar{V}_t(s,r):= V_t(s)$. Then $\bar{V}$ is a representation a representation of $\nu$. Indeed, for any $A \in \R^T$
$$
\bar{n}\circ \bar{V}^{-1}(A)=  \int_{S} \int_0^{g(s)^{-1}} \1_A(V(s)) \, g(s) dr n(ds) = n\circ V^{-1}(A)= \nu(A).
$$
A Poisson random measure $N := \sum_{j=1}^{\infty} \delta_{(\xi_j, \Gamma_j)}$ on $S \times \R_{+}$ has the intensity measure $g(s) n(ds) dr$, so the restriction of $N$ to $\bar{S}$ has the intensity $\bar{n}$. We have
\begin{equation} \label{}
  Y_t =  \int_{\bar{S}} V_t(s) \Big[\bar{N}(ds, dr) - \chi(V_t(s)) \bar{n}(ds,dr)\Big] + b(t)\, \quad a.s.
\end{equation}
Therefore, using \eqref{iso3} of Theorem \ref{t:iso2} we get
\begin{align*} 
  \E  \left[ F\( (V_t(\xi_0) + Y_t)_{t \in T} \) \right] &= \E \int_{S} \left[ F\( (V_t(s) + Y_t)_{t \in T} \) \right] q(s) n(ds) \\
  &=  \E \int_{S} \int_0^{g(s)^{-1}} \left[ F\( (V_t(s) + Y_t)_{t \in T} \) \right] g(s) q(s)  \, dr \, n(ds) \\
  &=  \E \int_{\bar{S}} \left[ F\( (V_t(s) + Y_t)_{t \in T} \) \right]  q(s)\1\{r \le g(s)^{-1}\} \,  \bar{n}(ds, dr) \\
  &= \E \left[ F\( (Y_t)_{t \in T} \); \, \bar{q}(\bar{N}) \right]\, ,
\end{align*}
where $\bar{q}(s,r)= q(s)\1\{r \le g(s)^{-1}\}$ and 
\begin{align*} 
  \bar{N}(\bar{q}) = \int_{\bar{S}} q(s)\1\{g(s) \le r^{-1}\} \, \bar{N}(ds,dr) = 
  \sum_{j=1}^{\infty} q(\xi_j)\1\{g(\xi_j) \le \Gamma_j^{-1}\} =Q.
\end{align*}
This proves \eqref{iso-s1}. 

Conversely, using \eqref{iso4} of Theorem \ref{t:iso2} we get
\begin{align*} 
  \E \left[ F\( (Y_t)_{t \in T}\); \, Q >0\right] &= \E \left[ F\( (Y_t)_{t \in T}\); \, \bar{N}(\bar{q}) >0\right] \\
  & =  \E \int_{\bar{S}}  \big[ F\left(\(V_t(s)+ Y_t\)_{t\in T} \right); \, (\bar{N}(\bar{q})+ \bar{q}(s,r))^{-1}\big] \, \bar{q}(s,r)\, \bar{n}(ds,dr) \\
  & =  \E \int_{S} \int_0^{g(s)^{-1}}  \big[ F\left(\(V_t(s)+ Y_t\)_{t\in T} \right); \, (Q+ q(s))^{-1}\big] \, q(s)\, dr\,   g(s) n(ds)\\
  &= \E \int_{S}  \big[ F\left(\(V_t(s)+ Y_t\)_{t\in T} \right); \, (Q+ q(s))^{-1}\big] \, q(s)\,  n(ds)  \, .
\end{align*}
The last statement of the theorem follows from the corresponding statement in Theorem \ref{t:iso2}. 
\qed

\bigskip

\np
{\bf Acknowledgments.} The author is grateful to Nathalie Eisenbaum, Haya Kaspi, Michael B. Marcus, and Jay Rosen for the stimulating discussions on various topics related to permanental processes. He is also grateful to the anonymous referee for useful remarks and calling our attention to \cite[Lemma 3.1]{Eisenbaum08}, where an abstract form of Dynkin's isomorphism was introduced and considered.

\medskip

\bibliographystyle{amsplain}

\bigskip

%

 \end{document}